\documentclass[11pt,leqno]{article}

\usepackage[active]{srcltx}

\usepackage{amsfonts,latexsym,amsmath,amssymb,amsthm}

\usepackage{fullpage}
\usepackage{dsfont}
\newtheorem{theorem}{Theorem}[section]
\newtheorem{lemma}[theorem]{Lemma}

\newtheorem{proposition}[theorem]{Proposition}

\newtheorem{remark}[theorem]{Remark}

\theoremstyle{definition}

\theoremstyle{remark}

\newtheorem*{note*}{Note}

\numberwithin{equation}{section}

\makeatletter
\newcommand{\rank}{\mathop{\operator@font rank}}
\newcommand{\conv}{\mathop{\operator@font conv}}
\newcommand{\vol}{\mathop{\operator@font vol}}
\newcommand{\onetagright}{\tagsleft@false}
\makeatother

\newcommand{\ls}{\leqslant}
\newcommand{\gr}{\geqslant}
\renewcommand{\epsilon}{\varepsilon}

\def\irr#1{{\rm Irr}(#1)}
\def\irrr#1#2 {\irr {#1 \mid #2}}

\usepackage{color}

\usepackage{hyperref}

\usepackage{setspace}
\begin{document}
\small

\title{\bf Moments of the Cram\'{e}r transform of log-concave probability measures}

\medskip

\author{Apostolos Giannopoulos and Natalia Tziotziou}

\date{}
\maketitle

\begin{abstract}\footnotesize Let $\mu$ be a centered log-concave probability measure on $\mathbb{R}^n$, and let $\Lambda_\mu^\ast$ denote its Cram\'er transform. We prove that
$\mathbb{E}_{\mu}\left[\exp\left(\frac{c_1}{n}\,\Lambda_{\mu}^{\ast}\right)\right] < \infty$ 
where $c_1>0$ is an absolute constant. In, particular, $\Lambda_{\mu}^{\ast}$ has finite moments of all orders. 
Our approach is based on a comparison of convex bodies naturally associated with $\mu$, and yields the bound
$\|\Lambda_\mu^\ast\|_{L^2(\mu)} \ls c_2\, n \ln n.$
The example of the uniform measure on the Euclidean ball shows that this growth in the dimension is optimal.  As an application, we obtain optimal dimension-dependent 
thresholds for the expected measure of random polytopes generated by independent log-concave samples.
\end{abstract}

\section{Introduction}

In this work we study a question that was left open in \cite{BGP-threshold} and concerns the existence of moments of
the Cram\'{e}r transform of log-concave probability measures. Let $\mu$ be a log-concave probability measure on ${\mathbb R}^n$.
The log-Laplace transform of $\mu$ is defined by
\begin{equation*}\Lambda_{\mu }(\xi )=\ln\left(\int_{{\mathbb R}^n}e^{\langle\xi
,z\rangle }d\mu(z)\right).\end{equation*}
Note that $\Lambda_{\mu}(0)=0$ and that $\Lambda_{\mu }$ is convex by H\"{o}lder's inequality. If we also
assume that the barycenter ${\rm bar}(\mu )$ of $\mu$ is at the origin then from Jensen's inequality
we see that $\Lambda_{\mu }(\xi )\gr 0$ for all $\xi $. One can also check that the set $A(\mu )=\{\Lambda_{\mu }<\infty\}$
is open and $\Lambda_{\mu }$ is $C^{\infty }$ and strictly convex on $A(\mu )$.
The Cram\'{e}r transform $\Lambda_{\mu}^{\ast }$ of $\mu $ is the Legendre transform of $\Lambda_{\mu}$, defined by
\begin{equation*}\Lambda_{\mu }^{\ast }(x)= \sup_{\xi\in {\mathbb R}^n} \left\{ \langle x, \xi\rangle - \Lambda_{\mu }(\xi )\right\},\end{equation*}
and plays a key role in the theory of large deviations (see \cite{Dembo-Zeitouni-book}).

Brazitikos, Pafis and the first named author develop in \cite{BGP-threshold} an approach to the ``threshold
problem" for the expected measure of random polytopes whose vertices have an arbitrary log-concave distribution.
To make this question precise, consider the random polytope $K_N={\rm conv}\{X_1,\ldots ,X_N\}$, where $N>n$ and
$X_1,X_2,\ldots $ are independent random vectors distributed according to $\mu$. Given $\delta\in\left(0,\frac{1}{2}\right)$, let
$$\varrho_1(\mu,\delta)=\sup\{r>0:\mathbb{E}_{\mu^N}[\mu(K_N)]\ls\delta\;\hbox{for all}\; N\ls e^r\}$$
and
$$\varrho_2(\mu,\delta)=\inf\{r>0:\mathbb{E}_{\mu^N}[\mu(K_N)]\gr 1-\delta\;\hbox{for all}\; N\gr e^r\},$$
where $\mu^N$ is the product measure $\mu\times\cdots\times\mu$ ($N$ times). Let also $\varrho(\mu,\delta)=\varrho_2(\mu,\delta)-\varrho_1(\mu,\delta)$. We say that $\mu$ exhibits a threshold around
$\tau>0$ for the expected measure ${\mathbb E}_{\mu^N}[\mu (K_N)]$ of $K_N$ if $\varrho_1(\mu,\delta)\ls \tau\ls\varrho_2(\mu,\delta)$
for all $\delta\in\left(0,\frac{1}{2}\right)$ and the ``window" $\varrho(\mu,\delta)$ is ``small" when compared
with $\tau$. The approach of \cite{BGP-threshold} establishes, under some conditions, a sharp threshold
with $\tau ={\mathbb E}_{\mu }(\Lambda_{\mu}^{\ast })$. One should of course show that
${\mathbb E}_{\mu }(\Lambda_{\mu}^{\ast })$ is finite, and in order to obtain a good bound for the window
of the threshold it is necessary to prove that the second moment of $\Lambda_{\mu}^{\ast }$ is finite and in fact
that the parameter $\beta(\mu)={\rm Var}_{\mu }(\Lambda_{\mu}^{\ast })/({\mathbb E}_{\mu }(\Lambda_{\mu }^{\ast }))^2$ is small
as the dimension $n$ increases to infinity, ideally that $\beta(\mu)=o_n(1)$ independently from $\mu$. More precisely,
combining the results of \cite{BGP-threshold} with subsequent estimates from \cite{Brazitikos-Chasapis-2024} we know that
if $\beta(\mu)\ls c_0$ then
$$\frac{\rho(\mu,\delta)}{\tau}\ls c_1\sqrt{\beta(\mu)/\delta}$$
for every log-concave probability measure $\mu$ on $\mathbb{R}^n$ and any $0<\delta<1/2$, where
$\tau=\mathbb{E}_{\mu}(\Lambda_{\mu}^{\ast})$ and $c_0,c_1>0$ are absolute constants.

It was proved in \cite{BGP-threshold} that if $K$ is a centered convex body of volume $1$ in ${\mathbb R}^n$
and $\mu$ is a centered $\kappa$-concave probability measure, where $\kappa\in (0,1/n]$, with ${\rm supp}(\mu)=K$
then ${\mathbb E}_{\mu}\big[\exp(\kappa \Lambda_{\mu}^{\ast }(x)/2)\big]<\infty $. In particular, this is true
for the Lebesgue measure $\mu_K$ on $K$ (with $\kappa =1/n$) and implies that $\Lambda_{\mu_K}^{\ast}$ has finite moments
of all orders. We obtain a similar integrability result in the broader setting of
log-concave probability measures.

\begin{theorem}\label{th:moments}For every centered log-concave probability measure $\mu$ on ${\mathbb R}^n$
we have that
\begin{equation}\label{eq:main-moments-intro}\int_{{\mathbb R}^n}\exp\left(\frac{c}{n}\Lambda_{\mu}^{\ast }(x)\right)\,d\mu(x)<\infty \end{equation}
where $c>0$ is an absolute constant. In particular, $\Lambda_{\mu}^{\ast}$ has finite moments of all orders.
\end{theorem}

We establish~\eqref{eq:main-moments-intro} with the constant $c=\frac{1}{16}$. A careful inspection of the proof of
Theorem~\ref{th:moments} shows that this constant can be improved, for instance, to any value
$c<\frac{1}{8}$. Moreover, we obtain the estimate
$$\int_{\mathbb{R}^n} \exp\!\left(\frac{\Lambda_\mu^{\ast}(x)}{16n}\right)\, d\mu(x)
\ls 2\,\exp\!\left(\frac{n\ln n}{16}\right).$$
A straightforward application of H\"older's inequality then yields the upper bound
$$\|\Lambda_\mu^{\ast}\|_{L^{\psi_1}(\mu)} \le C\, n^2 \ln n$$
for the Orlicz $\psi_1$-norm of $\Lambda_\mu^{\ast}$, corresponding to $\psi_1(t)=e^t-1$.

Theorem~\ref{th:moments} shows that $\|\Lambda_{\mu}^{\ast}\|_{L^2(\mu)}<+\infty$, and hence both $\mathbb{E}_{\mu}(\Lambda_{\mu}^{\ast})$
and $\beta(\mu)$ are finite for any centered log-concave probability measure $\mu$ on $\mathbb{R}^n$. In fact, we can
give the following upper bound for $\|\Lambda_{\mu}^{\ast}\|_{L^2(\mu)}$, which is optimal as one can check from the example
of the uniform measure on the Euclidean ball.

\begin{theorem}\label{th:small-moments}For every centered log-concave probability measure $\mu$ on ${\mathbb R}^n$
we have that
$$\|\Lambda_{\mu}^{\ast}\|_{L^2(\mu)}\ls cn\ln n$$
where $c>0$ is an absolute constant.
\end{theorem}

The proof of Theorem~\ref{th:moments} and Theorem~\ref{th:small-moments} is presented in Section~\ref{section-3}.
It is based on a number of intermediate results regarding certain families of convex bodies that are naturally associated
with a given centered log-concave probability measure $\mu$ on $\mathbb{R}^n$. First of all, it is natural to consider
the family $\{B_t(\mu)\}_{t>0}$ of level sets of $\Lambda_{\mu}^{\ast}$. For any $t>0$ we define
$B_t(\mu)=\{x\in\mathbb{R}^n:\Lambda_{\mu}^{\ast}(x)\ls t\}$. Then, the starting point for the
proof of, say, Theorem~\ref{th:small-moments} is the identity
$$\int_{{\mathbb R}^n}|\Lambda_{\mu}^{\ast}(x)|^2d\mu(x)=\int_0^{\infty}2t\,\mu(\{x:\Lambda_{\mu}^{\ast}(x)>t\})\,dt
=\int_0^{\infty}2t\,(1-\mu(B_t(\mu)))\,dt.$$
Therefore, we would like to know how fast $\mu(B_t)$ tends to $1$ as $t\to\infty$. We consider a second
family of convex bodies associated with $\mu$. For any $t>0$ we define
$R_t(\mu )=\{x\in {\mathbb R}^n:f_{\mu }(x)\gr e^{-t}f_{\mu }(0)\}$,
where $f_{\mu}$ is the density of $\mu$. An estimate for the measure of $R_t(\mu)$ is essentially
contained in Klartag's \cite[Lemma~5.2]{Klartag-2007clt}. His arguments lead to the estimate
$$\mu(R_t(\mu ))\gr 1-e^{-t/4}$$
for $t$ large enough with respect to $n$ (see Proposition~\ref{prop:r-2}). We compare $R_t(\mu)$
and $B_t(\mu)$ in Lemma~\ref{lem:level-B} where we show that for any $\delta\in (0,1)$ and $t\gr t_n:=n\ln n$ we have that
\begin{equation}\label{eq:intro-R-B}(1-\delta)R_t(\mu)\subseteq B_{g(t,\delta)}(\mu)\end{equation}
where $g(t,\delta)\ls 2t+n\ln(1/\delta)$. We combine the above with the useful observation (see Lemma~\ref{lem:2}) that
$$\mu((1+\delta)A)\ls e^{2n\delta }\mu(A)$$
for any $\delta>0$ and any Borel subset $A$ of $\mathbb{R}^n$. Choosing $\delta =t^{-4}$ we get that
$$\mu(B_{10t}(\mu))\gr 1-1/t^3$$
for all $t\gr t_n$, and Theorem~\ref{th:small-moments} follows. For the proof of Theorem~\ref{th:moments} we
apply a similar reasoning, this time using \eqref{eq:intro-R-B} with $\delta =\exp(-ct)$.

Besides the families $\{B_t(\mu)\}_{t>0}$ and $\{R_t(\mu)\}_{t>0}$ we discuss and compare several other interesting
families of convex bodies associated with $\mu$. These include the family $\{K_t(\mu)\}_{t>0}$ of K.~Ball's bodies, the family
$\{Z_t^+(\mu)\}_{t>0}$ of one-sided $L_t$-centroid bodies of $\mu$, and the family $\{T_t(\mu)\}_{t>0}$ of ``floating
bodies" of $\mu$, defined by $T_t(\mu)=\{x\in \mathbb{R}^n:\varphi_{\mu}(x)\gr e^{-t}\}$,
where $\varphi_{\mu }(x)=\inf\{\mu (H):H\;\hbox{is a closed half-space containing}\;x\}$
is Tukey's half-space depth function of $\mu$. This is carried out in Section~\ref{section-2}, where we introduce all of these
bodies and establish inclusion relations among them. Some relations  between specific pairs of these families have already been obtained in the
literature. For example, Paouris and Werner \cite[Theorem~2.2]{Paouris-Werner-2012} compare $T_t(\mu)$ with $Z_t(\mu)$, while 
\L ata\l a and Wojtaszczyk \cite[Proposition~3.2 and~3.5]{Latala-Wojtaszczyk-2008} compare $Z_t(\mu)$ with $B_t(\mu)$. In the
dual setting, Klartag and E.~Milman show in \cite[Lemma~2.3]{Klartag-EMilman-2012} that if $\Lambda_t(\mu)=\{\Lambda_{\mu}\ls t\}\cap (-\{\Lambda_{\mu}\ls t\})$ then
$\Lambda_t(\mu)\approx t(Z_t(\mu))^{\circ}$. However, sharp inclusions between pairs of these families, which would be useful for obtaining sharp
estimates on the rate at which the measure of these bodies converges to $1$ as $t\to\infty$, do not seem to appear
in the literature. We feel that a systematic and comprehensive collection of such results would also be useful
in other contexts as well.

We illustrate this point of view in Section~\ref{section-4} where we provide further applications of our
approach. Our first result concerns the question to obtain uniform upper and lower thresholds
for the expected measure of a random polytope defined as the convex hull of independent random points with a log-concave distribution. The general formulation of the problem is the following. Given a log-concave probability measure $\mu $ on ${\mathbb R}^n$
and the random polytope $K_N={\rm conv}\{X_1,\ldots ,X_N\}$ as above, our aim is to find a constant $N_1(n)$, depending only on $n$
and as large as possible, so that
$$\sup_{\mu }\Big(\sup\Big\{{\mathbb E}_{\mu^N}[\mu (K_N)]:N\ls N_1(n)\Big\}\Big)\longrightarrow 0$$ as $n\to\infty $, 
and a second constant $N_2(n)$, depending only on $n$ and as small as possible, so that
$$\inf_{\mu}\Big(\inf\Big\{{\mathbb E}_{\mu^N}[\mu (K_N)]:N\gr N_2(n)\Big\}\Big)\longrightarrow 1$$ as $n\to\infty $,
where the supremum and the infimum are over all log-concave probability measures. We shall call the first type of result a ``uniform upper threshold" and the second type a ``uniform lower threshold". Results of this form were obtained
by Chakraborti, Tkocz and Vritsiou in \cite{Chakraborti-Tkocz-Vritsiou-2021} for some
general families of distributions and by Brazitikos, Pafis and the first named author in \cite{BGP-depth} for the
class of all log-concave probability measures. An asymptotically optimal uniform upper threshold has been established
with $N_1(n)=\exp(cn)$, but the available uniform lower threshold $N_2(n)=\exp (C(n\ln n)^2u(n))$
where $u(n)$ is any function with $u(n)\to\infty $ as $n\to\infty $ is not optimal (see Section~\ref{section-4} for
a more detailed description of earlier works on this topic). Here we prove the next theorem.

\begin{theorem}\label{th:rough}There exists an absolute constant $C>0$ such that
$$\inf_{\mu }\Big(\inf\Big\{ {\mathbb E}_{\mu^N}\big[\mu (K_N)\big]: N\gr \exp (Cn\ln n)\Big\}\Big)\longrightarrow 1$$
as $n\to\infty $, where the first infimum is over all log-concave probability measures $\mu $ on ${\mathbb R}^n$.
\end{theorem}

The uniform lower threshold of Theorem~\ref{th:rough} is of the right order. An exponential in the dimension lower threshold is not possible in full generality; actually, in the case where $X_i$ are uniformly distributed in the Euclidean ball one can check
that $N\gr\exp (cn\ln n)$ points are needed so that the volume of a random $K_N$ will be significantly large.

The second result of Section~\ref{section-4} concerns the distribution of the half-space depth function $\varphi_{\mu}$.
Our starting point is the estimate
$$\exp(-c_1n)\ls {\mathbb E}_{\mu }(\varphi_{\mu }) \ls\exp(-c_2n)$$
from \cite{BGP-depth} which holds true for every log-concave probability measure $\mu$ on ${\mathbb R}^n$.
We consider negative moments of $\varphi_{\mu}$ and ask for the values of $p>0$ for which ${\mathbb E}_{\mu }(\varphi_{\mu }^{-p})<\infty$. Brazitikos and Chasapis have shown in \cite{Brazitikos-Chasapis-2024} that in the $1$-dimensional case one has
${\mathbb E}(\varphi_{\mu}^{-p})\ls 2^p/(1-p)$ for all $0<p<1$ and any probability measure $\mu$ on $\mathbb{R}$.
On the other hand, simple examples show that ${\mathbb E}_{\mu }(\varphi_{\mu }^{-p})$ may be infinite when $p\gr 1$.
We obtain the following general result on the range of values of $p>0$ for which
${\mathbb E}_{\mu }(\varphi_{\mu }^{-p})$ is finite.

\begin{theorem}\label{th:negative-phi}There exists an absolute constant $c>0$ such that
$$J_{\mu}(p):=\int_{{\mathbb R}^n}\frac{1}{\varphi_{\mu}^p(x)}\,d\mu (x)<\infty$$
for every log-concave probability measure $\mu$ on $\mathbb{R}^n$ and any $0<p\ls c/n$.
\end{theorem}

We establish the assertion of Theorem~\ref{th:negative-phi} with the constant $c=\frac{1}{32}$. A closer inspection
of the proof shows that $J_\mu(p)$ is finite for every $p>0$ such that the expectation
$\mathbb{E}\!\left[e^{2p\,\Lambda_\mu^{\ast}}\right]$ appearing in Theorem~\ref{th:moments} is finite.

\smallskip

Finally, we discuss the connection of the integrability properties of $\Lambda_{\mu}^{\ast}$
with the notion of affine surface area. In the language of the present work,
Sch\"{u}tt and Werner \cite{Schutt-Werner-1990} have proved that for every convex body $K$ of volume $1$ in ${\mathbb R}^n$ one has that
\begin{equation}\label{eq:constant-floating}\lim_{s\to\infty} e^{\frac{2s}{n+1}}(1-\mu_K(T_s(\mu_K)))=\frac{1}{2}\left(\frac{n+1}{\omega_{n-1}}\right)^{\frac{2}{n+1}}
{\rm as}(K)\end{equation} where $\mu_K$ is the Lebesgue measure on $K$, ${\rm as}(K)$ is the affine surface area of $K$ and $\omega_n$ is the volume of the
Euclidean unit ball $B_2^n$ (see Subsection~\ref{subsection-4.3} for a discussion).
In particular,
$$\sup\{e^{\frac{2s}{n+1}}(1-\mu_K(T_s(\mu_K))):s>0\}<\infty .$$
We provide an analogue of this fact in the more general setting of log-concave probability measures.

\begin{theorem}\label{th:T-measure}For every log-concave probability measure $\mu$ on $\mathbb{R}^n$,
$$\sup\{e^{\frac{1}{8}\frac{s}{n}}(1-\mu(T_s(\mu))):s>0\}\ls c_n,$$
where $c_n=\exp(n\ln n/8)$ is a constant depending only on $n$.
\end{theorem}

The constant $\tfrac{1}{8}$ obtained in Theorem~\ref{th:T-measure} is most likely non-optimal. Nevertheless, from the example 
of the uniform measure on a convex body, and in particular from~\eqref{eq:constant-floating}, we see that
Theorem~\ref{th:T-measure} cannot hold with a constant $c$ strictly larger than $2$.

\section{Families of convex bodies associated with log-concave measures}\label{section-2}

We work in $\mathbb{R}^n$ and use standard notation: $\langle\cdot ,\cdot\rangle $ is the usual inner product in ${\mathbb R}^n$
and the Euclidean norm is denoted by $|\cdot |$. We write $B_2^n$ for the Euclidean unit ball and $S^{n-1}$ for the unit sphere. Lebesgue measure in ${\mathbb R}^n$ is denoted by $|\cdot |$ and $\sigma $ is the rotationally invariant probability measure on $S^{n-1}$. The volume $\omega_n:=|B_2^n|$ of the Euclidean unit ball is equal to $\omega_n=\pi^{n/2}/\Gamma\left(\frac{n}{2}+1\right)=(c_n/n)^{n/2}$ for
some constant $c_n\approx 1$. Whenever we write $a\approx b$, we mean that there exist absolute constants $c_1,c_2>0$ such that $c_1a\ls b\ls c_2a$. We use the letters $c,c_1,c_2$ etc. to denote absolute positive
constants whose value may change from line to line.

We say that $K\subset {\mathbb R}^n$ is a convex body if it is compact, convex and has non-empty interior. We often consider bounded convex sets $K$ in ${\mathbb R}^n$ with $0\in {\rm int}(K)$; since the closure of such a set
is a convex body, we shall call these sets convex bodies too. We say that $K$ is centrally symmetric if $x\in K$ implies
that $-x\in K$ and that $K$ is centered if the barycenter ${\rm bar}(K)=\frac{1}{|K|}\int_Kx\,dx$ of $K$ is at the origin.
The radial function $\varrho_K$ of a convex body $K$ with $0\in {\rm int}(K)$ is the function
$\varrho_K(x)=\sup \{\lambda >0:\lambda x\in K\}$ defined for all $x\neq 0$, and
the support function of $K$ is the function $h_K(x) = \sup\{\langle x,y\rangle :y\in K\}$, $x\in {\mathbb R}^n$.
The polar body $K^{\circ }$ of a convex body $K$ in ${\mathbb R}^n$ with $0\in {\rm int}(K)$ is the convex body
\begin{equation*}
K^{\circ}:=\bigl\{y\in {\mathbb R}^n: \langle x,y\rangle \ls 1\;\hbox{for all}\; x\in K\bigr\}.
\end{equation*}
We say that a Borel probability measure $\mu$ on $\mathbb R^n$ is log-concave if $\mu(H)<1$ for every hyperplane $H$ in ${\mathbb R}^n$
(we then say that $\mu$ is full-dimensional) and $\mu(\lambda A+(1-\lambda)B) \gr \mu(A)^{\lambda}\mu(B)^{1-\lambda}$ for any pair of compact sets $A,B$ in ${\mathbb R}^n$
and any $\lambda \in (0,1)$. Borell \cite{Borell-1974} has proved that, under these assumptions, $\mu $ has a log-concave density $f_{{\mu }}$. Recall that a function $f:\mathbb R^n \rightarrow [0,\infty)$ is called log-concave if its support $\{f>0\}$ is a convex set in ${\mathbb R}^n$ and the restriction of $\ln{f}$ to it is concave. The Brunn-Minkowski inequality implies that if $K$ is a convex body in $\mathbb R^n$ then the indicator function $\mathds{1}_{K} $ of $K$ is the density of a log-concave measure, the Lebesgue measure on $K$.

We say that $\mu $ is even if $\mu (-B)=\mu (B)$ for every Borel subset $B$ of ${\mathbb R}^n$, and that $\mu $ is centered if the barycenter ${\rm bar}(\mu)$ of $\mu$
is at the origin, equivalently if
\begin{equation*}
\int_{\mathbb R^n} \langle x, \xi \rangle d\mu(x) = \int_{\mathbb R^n} \langle x, \xi \rangle f_{\mu}(x) dx = 0
\end{equation*} for all $\xi\in S^{n-1}$. We shall use the following result of Fradelizi from \cite{Fradelizi-1997}: if $\mu $ is a centered 
log-concave probability measure on ${\mathbb R}^n$ then
\begin{equation}\label{eq:frad-2}\|f_{\mu }\|_{\infty }\ls e^nf_{\mu }(0).\end{equation}
For any log-concave measure $\mu$ on ${\mathbb R}^n$ with density $f_{\mu}$, we define the isotropic constant of $\mu $ by
\begin{equation*}
L_{\mu }:=\left (\frac{\sup_{x\in {\mathbb R}^n} f_{\mu} (x)}{\int_{{\mathbb
R}^n}f_{\mu}(x)dx}\right )^{\frac{1}{n}} [\det \textrm{Cov}(\mu)]^{\frac{1}{2n}},\end{equation*}
where $\textrm{Cov}(\mu)$ is the covariance matrix of $\mu$ with entries
\begin{equation*}\textrm{Cov}(\mu )_{ij}:=\frac{\int_{{\mathbb R}^n}x_ix_j f_{\mu}
(x)\,dx}{\int_{{\mathbb R}^n} f_{\mu} (x)\,dx}-\frac{\int_{{\mathbb
R}^n}x_i f_{\mu} (x)\,dx}{\int_{{\mathbb R}^n} f_{\mu}
(x)\,dx}\frac{\int_{{\mathbb R}^n}x_j f_{\mu}
(x)\,dx}{\int_{{\mathbb R}^n} f_{\mu} (x)\,dx}.\end{equation*}
A log-concave probability measure $\mu $ on ${\mathbb R}^n$
is called isotropic if it is centered and $\textrm{Cov}(\mu )=I_n$,
where $I_n$ is the identity $n\times n$ matrix. Note that if $\mu$ is isotropic then $L_{\mu }=\|f_{\mu }\|_{\infty }^{1/n}$.

Let $\mu $ and $\nu$ be two log-concave probability measures on ${\mathbb R}^n$. Let $T:{\mathbb R}^n\to {\mathbb R}^n$ be a measurable function which is defined $\mu $-almost everywhere and satisfies
\begin{equation*}\nu (B)=\mu (T^{-1}(B))\end{equation*}
for every Borel subset $B$ of ${\mathbb R}^n$. We then say that $T$ pushes forward $\mu $ to $\nu $ and write $T_*\mu=\nu$. It is
easy to see that $T_*\mu =\nu $ if and only if for every bounded Borel
measurable function $g:{\mathbb R}^n\to {\mathbb R}$ we have
\begin{equation}\label{eq:push-forward}\int_{{\mathbb R}^n}g(x)d\nu (x)=\int_{{\mathbb R}^n}g(T(y))d\mu (y).\end{equation}
It is not hard to check that for every log-concave probability measure $\mu $ on $\mathbb{R}^n$ there exists
an invertible affine transformation $T$ such that the log-concave probability measure $T_{\ast }\mu $ is isotropic,
and $L_{T_{\ast}\mu}=L_{\mu}$.

The hyperplane conjecture asks if there exists an absolute constant $C>0$ such that
\begin{equation*}L_n:= \max\{ L_{\mu }:\mu\ \hbox{is an isotropic log-concave probability measure on}\ {\mathbb R}^n\}\ls C\end{equation*}
for all $n\gr 2$. The classical estimates $L_n\ls c\sqrt[4]{n}\ln n$ by Bourgain \cite{Bourgain-1991}
and, fifteen years later, $L_n\ls c\sqrt[4]{n}$ by Klartag \cite{Klartag-2006} remained the best known until
2020. In a breakthrough work, Chen \cite{C} proved that for any $\epsilon >0$
one has $L_n\ls n^{\epsilon}$ for all large enough $n$. This development was the starting point for
a series of important works, including a technical breakthrough by Guan \cite{Guan}, that culminated
in the final affirmative answer to the problem by Klartag and Lehec~\cite{KL} who showed that $L_n\ls C$.
Shortly thereafter, Bizeul~\cite{Bizeul-2025} provided an alternative proof.

In the next three subsections, we introduce and compare the families of convex bodies that are associated with
a centered log-concave probability measure $\mu$ on ${\mathbb R}^n$ and have been mentioned in the introduction.
Some of these families can be defined for any probability measure and the assumption that $\mu$ is log-concave
or centered is unnecessary for some of the results that we present. However, for simplicity, we chose to define
and state everything in the setting of centered log-concave probability measures.

We refer to Schneider's book \cite{Schneider-book} for basic facts from the Brunn-Minkowski theory and to the book
\cite{AGA-book} for basic facts from asymptotic convex geometry. We also refer to \cite{BGVV-book} for more information on isotropic
convex bodies and log-concave probability measures.

\subsection{Level sets of the density and K.~Ball's bodies}\label{subsection-2.1}

Let $\mu$ be a centered log-concave probability measure on $\mathbb{R}^n$. For every $t\gr 1$ we consider the convex set
$$R_t(\mu )=\{x\in {\mathbb R}^n:f_{\mu }(x)\gr e^{-t}f_{\mu }(0)\}.$$
Using the log-concavity of $f_{\mu }$ we easily check that $R_t(\mu )$ is convex. Note also that $0\in {\rm int}(R_t(\mu ))$.
To show that $R_t(\mu)$ is bounded, we recall that since $f_{\mu}$ is log-concave and has finite positive integral we
have that there exist constants $A,B>0$ such that
\begin{equation}\label{eq:A-B}f_{\mu}(x)\ls Ae^{-B|x|}\end{equation}
for all $x\in {\mathbb R}^n$ (see \cite[Lemma~2.2.1]{BGVV-book}). Therefore, if $x\in R_t(\mu)$ we get
that $|x|\ls\frac{1}{B}\big(\ln (A/f_{\mu}(0))+t\big)$. Another consequence of \eqref{eq:A-B} is that $f_{\mu}$ has
finite moments of all orders.

The next proposition, which will be very useful for us, shows that the measure of $R_t(\mu)$ increases to $1$
exponentially fast as $t\to\infty $.

\begin{proposition}\label{prop:r-2}For every $t\gr 5(n-1)$ we have that $\mu(R_t(\mu ))\gr 1-e^{-t/4}$.
\end{proposition}

The proof is based on some one-dimensional considerations. Let $g:[0,\infty)\to[0,\infty)$ be a log-concave function. It is proved in \cite[Lemma~4.3]{Klartag-2007clt} that for every $m>0$ the equation $(\ln g)^{\prime}(r)=-\frac{m}{r}$ has a unique
solution $r_m>0$. Moreover, for every $0\ls r\ls r_m$ one has that
\begin{equation}\label{eq:r-1}g(r)\gr e^{-m}g(0).\end{equation}
Since $(\ln g)^{\prime}$ is decreasing on the interval where it is defined, we see that $(\ln g)^{\prime}(r)\ls -m/r_m$
for every $r>r_m$ such that $g(r)>0$. Therefore, for every $\alpha\gr 1$ and every $r\gr\alpha r_m$ we have that
\begin{equation}\label{eq:r-2}g(r)\ls e^{-(\alpha-1)m}g(r_m).\end{equation}
The next lemma is essentially contained in Klartag's \cite[Lemma~5.2]{Klartag-2007clt}.

\begin{lemma}\label{lem:r-1}Let $m> 0$, $\alpha\gr 5$ and let $g:[0,\infty)\to [0,\infty)$ be a log-concave
function with finite positive integral. Let $\varrho_{\alpha m}:=\sup\{r>0:g(r)\gr e^{-\alpha m}g(0)\}$. Then,
$$\int_0^{\varrho_{\alpha m}}r^{m}g(r)\,dr\gr\left(1-e^{-\alpha m/4}\right)\int_0^{\infty}r^{m}g(r)\,dr.$$
\end{lemma}

\begin{proof}We may assume that $\int_0^{\infty}r^{m}g(r)\,dr=1$. For $r>0$ we define
$$\varphi(r)=r^{m}g(r)\quad\hbox{and}\quad \Phi(r)=\int_0^r\varphi(u)\,du.$$
Note that $\varphi$ is a log-concave function with $\int_0^{\infty}\varphi(r)\,dr=1$.
Define $r_m$ as the unique solution of $(\ln g)^{\prime}(r)=-m/r$. Then, equivalently, we have
that $\varphi^{\prime}(r_m)=0$ and since $\varphi$ is log-concave and $\varphi(r_m)>0$ we conclude
that $\varphi(r_m)=\max(\varphi)$. From \eqref{eq:r-1} we know that
$$g(r)\gr e^{-m}g(0)$$
for all $0\ls r\ls r_m$. It follows that if $M:=g(r_m)$ then, for every $r>0$ that satisfies
$g(r)\gr e^{-(\alpha-1)m}M$ we have that $g(r)\gr e^{-\alpha m}g(0)$, and hence
$$\varrho_{\alpha m}\gr r^{\prime}:=\sup\{r>0:g(r)\gr e^{-(\alpha -1)m}M\}.$$
From the definition of $r^{\prime}$ it is clear that $r_m\ls r^{\prime}$.
Also, since $g(r^{\prime})=e^{-(\alpha-1)m}M$, from \eqref{eq:r-2} we see that $r^{\prime}\ls\alpha r_m$. It follows that
\begin{equation*}
\varphi(r^{\prime}) =\varphi(r_m)\left(\frac{r^{\prime}}{r_m}\right)^{m}\frac{g(r^{\prime})}{M}\ls
\varphi(r_m)\alpha^{m}e^{-(\alpha-1)m}\ls \varphi(r_m)e^{-\alpha m/4}=e^{-\alpha m/4}\max(\varphi).
\end{equation*}
The last inequality follows from the fact that $3\alpha \gr 4(\ln \alpha +1)$ when $\alpha\gr 5$.
It is not hard to check that $\psi(s)=\varphi(\Phi^{-1}(s))$ is concave on $(0,1)$.  Since $\varphi$
attains its maximum at $r_m$, we have that $\psi$ attains its maximum at $\Phi(r_m)$. Then, for any
$0<\varepsilon\ls 1$, using the concavity of $\psi$ we can check that if
$s\gr \Phi(r_m)$ and $\psi(s)\ls\varepsilon\cdot\max(\psi)$
we must have that $s\gr 1-\varepsilon$. It follows that if $r\gr r_m$ and
$\varphi(r)\ls\varepsilon\cdot\max(\psi)=\varepsilon\cdot\max(\varphi)$ then
$\Phi(r)\gr 1-\varepsilon$. Since $r^{\prime}\gr r_m$ and $r^{\prime}$
satisfies $\varphi(r^{\prime})\ls e^{-\alpha m/4}\max(\varphi)$, we conclude that
$$\Phi(r^{\prime})\gr 1-e^{-\alpha m/4}.$$
The result follows from the fact that $\varrho_{\alpha m}\gr r^{\prime}$.
\end{proof}

Lemma~\ref{lem:r-1} immediately implies Proposition~\ref{prop:r-2}.

\begin{proof}[Proof of Proposition~$\ref{prop:r-2}$]Fix $\xi\in S^{n-1}$ and consider the log-concave integrable function $g:[0,\infty)\to[0,\infty)$
defined by $g(r)=f_{\mu}(r\xi)$. Applying Lemma~\ref{lem:r-1} with $m=n-1$ and $\alpha=t/m$ we see that
$$\int_0^{\varrho_{R_t(\mu)}(\xi)}r^{n-1}f_{\mu}(r\xi)\,dr
\gr \left(1-e^{-t/4}\right)\int_0^{\infty}r^{n-1}f_{\mu}(r\xi)\,dr.$$
Then, integration in polar coordinates shows that
\begin{align*}\mu(R_t(\mu )) &=n\omega_n\int_{S^{n-1}}\int_0^{\varrho_{R_t(\mu)}(\xi)}r^{n-1}f_{\mu}(r\xi)\,dr\,d\sigma(\xi)\\
&\gr \left(1-e^{-t/4}\right)\,n\omega_n\int_{S^{n-1}}\int_0^{\infty}r^{n-1}f_{\mu}(r\xi)\,dr\,d\sigma(\xi)=1-e^{-t/4},
\end{align*}
which is the assertion of the proposition.
\end{proof}

\begin{remark}\label{rem:referee}\rm The referee of this article has very kindly communicated to us the following very short and elegant proof
of Proposition~\ref{prop:r-2}. Consider the convex function $\psi$ defined by $e^{-\psi(x)}=f_{\mu}(x)/f_{\mu}(0)$. Note that $\psi(0)=0$.
Then
\begin{align*}
\int_{\mathbb{R}^n}e^{\psi(x)/2}d\mu(x) &=f_{\mu}(0)\int_{\mathbb{R}^n}e^{-\psi(x)/2}dx=f_{\mu}(0)\int_{\mathbb{R}^n}e^{-(\psi(x)+\psi(0))/2}dx\\
&\ls f_{\mu}(0)\int_{\mathbb{R}^n}e^{-\psi(x/2)}dx=2^nf_{\mu}(0)\int_{\mathbb{R}^n}e^{-\psi(x)}dx=2^n.
\end{align*}
For any $t>0$ we have $R_t(\mu)=\{x:\psi(x)\ls t\}$. From Markov's inequality we get
$$1-\mu(R_t(\mu))=\mu(\{x:\psi(x)>t\})\ls e^{-t/2}\int_{\mathbb{R}^n}e^{\psi(x)/2}d\mu(x)\ls 2^ne^{-t/2}.$$
If $t\gr (4\ln 2)n$ then $2^ne^{-t/2}\ls e^{-t/4}$, and this implies that
$$\mu(R_t(\mu))\gr 1-e^{-t/4},\qquad\text{for all}\;t\gr (4\ln 2)n.$$
\end{remark} 

A second family of convex bodies associated with a centered log-concave probability measure $\mu$ on
$\mathbb{R}^n$ was introduced by K.~Ball, who also established their convexity in \cite{Ball-1988}: for every $t>0$, we define
\begin{equation*}
K_t(\mu)=\left\{x\in {\mathbb R}^n : \int_0^\infty r^{t-1}f_\mu(rx)\,
dr\gr \frac{f_{\mu }(0)}{t} \right\}.
\end{equation*}
From the definition it follows that the radial function of $K_t(\mu )$ is given by
\begin{equation}\label{eq:radial-Kp}
\varrho_{K_t(\mu )}(x)=\left (\frac{1}{f_{\mu }(0)}\int_0^{\infty}tr^{t-1}f_{\mu }(rx)\,dr\right )^{1/t}\end{equation}
for $x\neq 0$. For every $0 <t\ls s$ we have that
\begin{equation}\label{eq:inclusions-Kp}\frac{\Gamma(t+1)^{\frac{1}{t}}}{\Gamma(s+1)^{\frac{1}{s}}} K_s(\mu )\subseteq
K_{t}(\mu )\subseteq e^{\frac{n}{t}-\frac{n}{s}}K_{s}(\mu ).
\end{equation} A proof is given in \cite[Proposition~2.5.7]{BGVV-book}.

An immediate consequence of the definitions is the next inclusion between the bodies $K_t(\mu)$ and $R_t(\mu)$.

\begin{proposition}\label{prop:r-3}For every $s>t$ we have that
$$R_t(\mu )\subseteq e^{t/s}K_s(\mu ).$$
\end{proposition}

\begin{proof}Let $\xi\in S^{n-1}$ and define $g:[0,\infty )\to [0,\infty )$ by
$g(r)=f_{\mu }(r\xi )$. By the definition of $K_s(\mu )$ we have
$$[\varrho_{K_s(\mu )}(\xi )]^s=\frac{s}{f_{\mu }(0)}\int_0^{\infty }r^{s-1}g(r)dr.$$
Since $g(r)\gr e^{-t}f_{\mu}(0)$ for all $0\ls r\ls\varrho_{R_t(\mu)}(\xi)$ we see that
$$\int_0^{\infty }r^{s-1}g(r)dr\gr \int_0^{\varrho_{R_t(\mu)}(\xi)}r^{s-1}g(r)dr
\gr e^{-t}f_{\mu}(0)\int_0^{\varrho_{R_t(\mu)}(\xi)}r^{s-1}dr=\frac{f_{\mu}(0)}{s}e^{-t}[\varrho_{R_t(\mu )}(\xi )]^s.$$
It follows that $\varrho_{K_s(\mu )}(\xi )\gr e^{-t/s}\varrho_{R_t(\mu )}(\xi )$. Since $\xi\in S^{n-1}$
was arbitrary, this shows that $R_t(\mu )\subseteq e^{t/s}K_s(\mu )$.
\end{proof}

In the opposite direction we have the next result.

\begin{proposition}\label{prop:r-4}Let $n\gr 3$. For every $t\gr 2n$ and any $\alpha\gr 5$ we have that
$$R_{\alpha t}(\mu )\supseteq \left(1-\frac{2n}{t}\right)K_t(\mu ).$$
Under the additional assumption that $\mu$ is even, we have that
$$R_{\alpha t}(\mu )\supseteq \left(1-e^{-\alpha t/5}\right)K_t(\mu ).$$
\end{proposition}

\begin{proof}Given any $\xi\in S^{n-1}$ consider the log-concave function $g:[0,\infty )\to [0,\infty )$ defined by
$g(r)=f_{\mu }(r\xi )$. Applying Lemma~\ref{lem:r-1} with $m=t-1$ we see that
$$\int_0^{\varrho_{R_{\alpha t}(\mu )}(\xi )}r^{t-1}g(r)dr\gr (1-e^{-\alpha t/5})\int_0^{\infty }r^{t-1}g(r)dr.$$
By the definition of $K_t(\mu )$ we have
$$\int_0^{\infty }r^{t-1}g(r)dr=\frac{f_{\mu }(0)}{t}[\varrho_{K_t(\mu )}(\xi )]^t.$$
On the other hand,
$$\int_0^{\varrho_{R_{\alpha t}(\mu )}(\xi )}r^{t-1}g(r)dr\ls \|f_{\mu}\|_{\infty }\int_0^{\varrho_{R_{\alpha t}(\mu )}(\xi )}r^{t-1}dr
=\frac{\|f_{\mu}\|_{\infty }}{t}[\varrho_{R_{\alpha t}(\mu )}(\xi )]^t.$$
Using also the fact that $\|f_{\mu}\|_{\infty }\ls e^nf_{\mu }(0)$ from \eqref{eq:frad-2}, we get
\begin{equation}\label{eq:r-3}e^n[\varrho_{R_{\alpha (t-1)}(\mu )}(\xi )]^t
\gr (1-e^{-\alpha t/5})[\varrho_{K_t(\mu )}(\xi )]^t\end{equation}
and the result follows because
$$e^{-n/t}(1-e^{-\alpha t/5})^{1/t}\gr (1-n/t)(1-e^{-\alpha t/5})\gr
(1-n/t)^2\gr 1-2n/t.$$ When $\mu$ is even, the term $e^n$ does not appear
in \eqref{eq:r-3}, and thus we obtain an improved estimate. \end{proof}

\subsection{Level sets of the Cram\'{e}r transform and centroid bodies}\label{subsection-2.2}

Let $\mu$ be a centered log-concave probability measure on $\mathbb R^n$. For any $t\gr 1$ we define the
$L_t$-centroid body $Z_t(\mu)$ of $\mu $ as the convex body whose support function is
\begin{equation*} h_{Z_t(\mu)}(y):=\left(
\int_{\mathbb R^n} |\langle x,y\rangle|^t f_{\mu}(x)dx \right)^{1/t},\qquad y\in {\mathbb R}^n.
\end{equation*} Note that $Z_t(\mu )$ is always centrally symmetric,
and $Z_t(T_{\ast}\mu)=T(Z_t(\mu ))$ for every $T\in GL(n)$ and $t\gr 1$.
We shall also use the fact that if $\mu$ is isotropic then
$Z_2(\mu)=B_2^n$. A variant of the $L_t$-centroid bodies of $\mu $ is defined as follows. For every $t\gr 1$ we consider the convex body
$Z_t^+(\mu )$ with support function
\begin{equation*}h_{Z_t^+(\mu )}(y)=\left (\int_{{\mathbb R}^n}\langle
x,y\rangle_+^tf_{\mu }(x)dx\right )^{1/t},\qquad y\in {\mathbb R}^n,\end{equation*} where $a_+=\max
\{a,0\}$. When $f_{\mu }$ is even, we have that $Z_t^+(\mu )=2^{-1/t}Z_t(\mu )$. In any case, it is clear that
$Z_t^+(\mu )\subseteq Z_t(\mu )$. One can also check that if $1\ls t<s$ then
\begin{equation}\label{eq:regularity}\left(\frac{4}{e}\right)^{\frac{1}{t}-\frac{1}{s}}Z_t^+(\mu )\subseteq Z_s^+(\mu )
\subseteq c_1\left(\frac{4(e-1)}{e}\right)^{\frac{1}{t}-\frac{1}{s}}\frac{s}{t}Z_t^+(\mu ).\end{equation}
(for a proof see \cite{Guedon-EMilman-2011}, where the family of bodies $\tilde{Z}_t^+(\mu)=2^{1/t}Z_t^+(\mu)$ is considered).

For every $t\gr 1$ we define \begin{equation*}M^+_t(\mu) := \left\{v\in {\mathbb R}^n : \int_{{\mathbb R}^n} \langle v, x\rangle_+^td\mu(x)\ls 1\right\}.\end{equation*}Note that
\begin{equation*}Z^+_t(\mu) := (M^+_t(\mu))^{\circ}.\end{equation*}For every $t>0$ we also define
\begin{equation*}B_t(\mu):=\{v\in{\mathbb R}^n:\Lambda^{\ast}_{\mu}(v)\ls t\}.\end{equation*}
For all $s\gr t$ we have $M^+_s(\mu)\subseteq M^+_t(\mu)$ and $Z_t^+(\mu) \subseteq Z_s^+(\mu)$.
Moreover, since $\mu $ is centered we have $\Lambda^{\ast }_{\mu }(0)=0$ by Jensen's inequality, and the convexity
of $\Lambda^{\ast}_{\mu }$ implies that $B_t(\mu) \subseteq B_s(\mu) \subseteq
\frac{s}{t}B_t(\mu)$ for all $s\gr t>0$. A proof of all these assertions
can be found in \cite{BGVV-book}.

The next proposition, which is a variant of \cite[Proposition~3.2]{Latala-Wojtaszczyk-2008} of \L ata\l a and Wojtaszczyk,
provides an inclusion relation between the bodies $Z_t^+(\mu)$ and $B_s(\mu)$. In fact, the assumption that $\mu$ is log-concave
is not required for the proof below.

\begin{proposition}\label{prop:1}There exists $t_0\gr 1$ such that, for every centered log-concave probability measure $\mu $ on ${\mathbb R}^n$ and 
every $s\gr t\gr t_0$,
\begin{equation*}Z_t^+(\mu) \subseteq \left(1+\frac{2\ln s}{s}\right)B_s(\mu).\end{equation*}
\end{proposition}

\begin{proof}Let $v\in Z_t^+(\mu)$. We shall show that $\Lambda^{\ast }_{\mu }(v/c_s)\ls s$, where
$c_s\ls 1+\frac{2\ln s}{s}$. To this end, we should check that
\begin{equation*}\frac{1}{c_s}\langle u, v\rangle-\Lambda_\mu(u) \ls s\end{equation*}
for all $u\in{\mathbb R}^n$. For every $u\in{\mathbb R}^n$ we define $\beta_u$ by the equation
\begin{equation*}\int_{{\mathbb R}^n}\langle u, x\rangle_+^td\mu(x) = \beta_u^t.\end{equation*}
Then, $u/\beta_u\in M^+_t(\mu)$ and since $Z_t^+(\mu )=(M_t^+(\mu ))^{\circ }$
we have that $\langle u,v\rangle \ls \beta_u$. Note that, by H\"{o}lder's inequality, for all $s\gr t$ we have that
\begin{equation}\label{eq:1-1}\int_{{\mathbb R}^n}\langle u, x\rangle_+^sd\mu(x)\gr \left(\int_{{\mathbb R}^n}\langle u, x\rangle_+^td\mu(x)\right)^{s/t}=\beta_u^s.\end{equation}
For any $0<\delta<\beta_u$, using \eqref{eq:1-1} we see that
\begin{align*}\int_{{\mathbb R}^n} e^{\frac{\delta}{\beta_u}\langle u,x\rangle }d\mu(x) &\gr  \int_{{\mathbb R}^n} e^{\frac{\delta}{\beta_u}\langle
u,x\rangle}\mathds{1}_{\{x:\langle u,x\rangle \gr 0\}}(x)\,d\mu(x)
\gr \int_{{\mathbb R}^n} \frac{\delta^s\langle u, x\rangle_+^s}{s!\beta_u^s} \mathds{1}_{\{x:\langle
u,x\rangle \gr 0\}}(x)\,d\mu(x) \\
&=\int_{{\mathbb R}^n} \frac{\delta^s\langle u, x\rangle_+^s}{s!\beta_u^s}\,d\mu(x)
\gr\frac{\delta^s}{s!}.\end{align*}
So, $\Lambda_{\mu}(\delta u/\beta_u)\gr \ln\left(\frac{\delta^s}{s!}\right)$, and using the convexity of $\Lambda_{\mu}$
and the fact that $\Lambda_{\mu}(0)=0$ we obtain
\begin{equation*}\Lambda_{\mu}(u) \gr \frac{\beta_u}{\delta}\Lambda_{\mu}(\delta u/\beta_u)
\gr \beta_u\,\frac{1}{\delta}\ln\left(\frac{\delta^s}{s!}\right).\end{equation*} Therefore, for any $c>0$ we get
\begin{equation}\label{eq:1-3}\frac{1}{c}\langle u,v\rangle-\Lambda_{\mu}(u) \ls
\beta_u \left[\frac{1}{c}-\frac{1}{\delta}\ln\left(\frac{\delta^s}{s!}\right)\right].\end{equation}
The function $g_s(\delta)=\frac{1}{\delta}\ln\left(\frac{\delta^s}{s!}\right)$ attains its maximum value
at $\delta_s$ where $\delta_s$ satisfies $s=\ln\left(\frac{\delta^s}{s!}\right)$, i.e. $\delta_s=e(s!)^{1/s}$.
This maximum value is equal to $\max(g_s)=\frac{s}{\delta_s}=\frac{s}{e(s!)^{1/s}}$. So, if we choose $\delta=\delta_s$
and $c_s=\frac{1}{\max(g_s)}$, from \eqref{eq:1-3} we see that
\begin{equation}\label{eq:1-4}\frac{1}{c_s}\langle u,v\rangle-\Lambda_{\mu}(u)\ls 0\end{equation}
for all $u\in\mathbb{R}^n$ that satisfy $\beta_u>\delta_s$. On the other hand, since $\mu$ is centered, we have that
$\Lambda_{\mu}(u)\gr 0$ for all $u\in\mathbb{R}^n$, and hence, for all $u\in\mathbb{R}^n$ that satisfy $\beta_u\ls\delta_s$ we have that
\begin{equation}\label{eq:1-5}\frac{1}{c_s}\langle u,v\rangle -\Lambda_\mu(u) \ls \frac{1}{c_s}\beta_u\ls \frac{\delta_s}{c_s}=s.\end{equation}
This shows that
$$\Lambda_{\mu}^{\ast}(v/c_s)=\sup\left\{\frac{1}{c_s}\langle u,v\rangle -\Lambda_{\mu}(u):u\in\mathbb{R}^n\right\}
\ls s,$$
or equivalently $v\in c_sB_s$. We have thus proved that $Z_t^+\subseteq c_sB_s$ for all $s\gr t$ and it remains to
estimate the constant $c_s$. By Stirling's formula,
$$c_s=\frac{e}{s}\,(s!)^{1/s}\sim \frac{e}{s}\left(\frac{s}{e}\right)(2\pi s)^{\frac{1}{2s}}=(2\pi s)^{\frac{1}{2s}}\sim
1+\frac{\ln(2\pi s)}{2s}$$
as $s\to\infty$. It follows that if $s\gr t\gr t_0$ (where $t_0$ is a large enough constant) then $c_s\ls 1+\frac{2\ln s}{s}$.
\end{proof}

The next proposition establishes a reverse inclusion between the bodies $B_t(\mu )$ and $Z_s^+(\mu )$
(again, a variant of this result appears in \cite[Proposition~3.5]{Latala-Wojtaszczyk-2008}).

\begin{proposition}\label{prop:B<Z}Let $\mu $ be a centered log-concave probability measure on $\mathbb{R}^n$.
For any $t\gr 1$ and any $\delta\in (0,1)$ we have that
\begin{equation*}B_t(\mu) \subseteq (1+\delta)Z_{c_1t/\delta}^+(\mu)\end{equation*}
where $c_1>0$ is an absolute constant.
\end{proposition}

\begin{proof}Let $\delta\in (0,1)$ and $t\gr 1$. If $u\in M^+_s(\mu)$ then H\"{o}lder's inequality shows that
$\|\langle u,\cdot\rangle_+\|_k\ls \|\langle u,\cdot\rangle_+\|_s\ls 1$ for all $k\ls s$, and \eqref{eq:regularity}
implies that $\|\langle u,\cdot\rangle_+\|_k\ls \frac{ck}{s}\|\langle u,\cdot\rangle_+\|_s\ls \frac{ck}{s}$
for all $k>s$, where $c>0$ is an absolute constant. Since $\frac{k}{(k!)^{1/k}}\to e$, we may choose an absolute constant $\gamma>0$ small enough so that $\frac{c\gamma k}{(k!)^{1/k}}\ls\frac{1}{2}$ for all $k\gr 1$.  It follows that
\begin{align*}
\int_{{\mathbb R}^n} e^{\langle \gamma su,x\rangle_+}d\mu(x)
&=\sum_{k=0}^{\infty}\frac{1}{k!}\int_{{\mathbb R}^n}\langle \gamma su, x\rangle_+^kd\mu(x)
\ls \sum_{k\ls s}\frac{(\gamma s)^k}{k!}+\sum_{k>s}\frac{(\gamma s)^k}{k!}\left (\frac{ck}{s}\right )^k\\
&\ls e^{\gamma s}+\sum_{k>s}\frac{1}{2^k}\ls e^{\gamma s}+1\ls e^{\gamma s+1}\end{align*}
if $s\gr s_0$. Therefore, for any $u\in M^+_s(\mu)$ we get $\Lambda_{\mu}(\gamma su)\ls\gamma s+1$.

Now, let $v\notin (1+\delta)Z^+_s(\mu)$. We can find $u\in M^+_s(\mu)$ such that $\langle v,u\rangle> 1+\delta $ and then
\begin{equation*}\Lambda^{\ast }_{\mu }(v) \gr \langle v,\gamma su\rangle
-\Lambda_{\mu }(\gamma su)> (1+\delta)\gamma s -\gamma s-1=\delta\gamma s-1>t\end{equation*}
if we assume that $s\gr \frac{2t}{\gamma\delta}$. Therefore, $v\notin
B_t(\mu )$. This shows that $B_t(\mu)\subseteq (1+\delta)Z^+_{c_1t/\delta} (\mu)$, where $c_1=2/\gamma$. \end{proof}

\subsection{Level sets of the Cram\'{e}r transform and floating bodies}\label{subsection-2.3}

Let $\mu $ be a probability measure on ${\mathbb R}^n$. For any $x\in {\mathbb R}^n$ we denote by ${\cal H}(x)$ the set
of all closed half-spaces $H$ of ${\mathbb R}^n$ containing $x$. The function
$$\varphi_{\mu }(x)=\inf\{\mu (H):H\in {\cal H}(x)\}$$
is called Tukey's half-space depth. This notion was introduced by Tukey in \cite{Tukey-1975} for data sets as a
measure of centrality for multivariate data; see the survey article of Nagy, Sch\"{u}tt and Werner
\cite{Nagy-Schutt-Werner-2019} for an overview, connections
with convex geometry and references. Note that the infimum in the definition of $\varphi_{\mu}(x)$
is determined by those closed half-spaces $H$ for which $x$ lies on the boundary $\partial (H)$ of $H$.
For every $s\gr 1$ we define the set
$$T_s(\mu)=\{x\in \mathbb{R}^n:\varphi_{\mu}(x)\gr e^{-s}\}.$$
Note that $T_s(\mu)$ is convex: if $x,y\in T_s(\mu)$ then for any $z\in [x,y]$ and any $H\in {\cal H}(z)$ we
have that either $x$ or $y$ belongs to $H$, and hence $\mu(H)\gr e^{-s}$, therefore $\varphi_{\mu}(z)\gr e^{-s}$.

Let $x\in {\mathbb R}^n$. For any $\xi \in {\mathbb R}^n$ the half-space
$\{z:\langle z-x,\xi  \rangle \gr 0\}$ is in ${\cal H}(x)$, therefore
\begin{equation*}
\varphi_{\mu } (x) \ls \mu (\{z:\langle z,\xi  \rangle \gr \langle x,\xi \rangle\} )
\ls e^{-\langle x,\xi \rangle }{\mathbb E}_{\mu }\big(e^{\langle z,\xi \rangle }\big)=\exp \big(-[\langle x,\xi \rangle -\Lambda_{\mu }(\xi )]\big),\end{equation*}
and taking the infimum over all $\xi \in {\mathbb R}^n$ we see that
$\varphi_{\mu } (x)\ls\exp (-\Lambda_{\mu }^{\ast }(x))$. An immediate consequence of this inequality is the inclusion
\begin{equation}\label{eq:floating-1}T_s(\mu)\subseteq B_s(\mu).\end{equation}
Indeed, if $x\in T_s(\mu)$ then $e^{-s}\ls\varphi_{\mu}(x)\ls\exp (-\Lambda_{\mu }^{\ast }(x))$, which shows that
$\Lambda_{\mu}^{\ast}(x)\ls s$, and hence $x\in B_s(\mu)$.

The next proposition shows that if $\mu$ is log-concave then a reverse inclusion holds in full generality,
at least if $s$ is large enough.

\begin{proposition}\label{prop:floating-1}There exists $s_0\gr 1$ such that, for every centered log-concave
probability measure $\mu$ on $\mathbb{R}^n$ and any $s\gr s_0$,
$$T_s(\mu)\subseteq B_s(\mu)\subseteq T_{s+3\ln s}(\mu).$$
\end{proposition}

\begin{proof}We shall use a result of Brazitikos and Chasapis from \cite{Brazitikos-Chasapis-2024}: For every $x\in {\rm supp}(\mu)$
and any $\varepsilon\in (0,1)$ we have that
$$\Lambda_{\mu}^{\ast}(x)\gr (1-\varepsilon)\ln\left(\frac{1}{\varphi_{\mu}(x)}\right)
+\ln\left(\frac{\varepsilon}{2^{1-\varepsilon}}\right)=\ln\left(\frac{\varepsilon}{(2\varphi_{\mu}(x))^{1-\varepsilon}}\right).$$
Let $x\in B_s(\mu)$ and assume that $\varphi_{\mu}(x)<e^{-s-3\ln s}$. Then,
$$2\varphi_{\mu}(x)\ls e^{-s-3\ln s+\ln 2}<e^{-s-2\ln s}$$
provided that $s\gr 2$, and hence
$$s\gr \Lambda_{\mu}^{\ast}(x)\gr \ln\left(\varepsilon e^{(1-\varepsilon)(s+2\ln s)}\right)
=\ln\varepsilon +(1-\varepsilon)(s+2\ln s)$$
for every $\varepsilon\in (0,1)$ and any $s\gr 2$. The function $f(\varepsilon)=\ln\varepsilon +(1-\varepsilon)(s+2\ln s)$
attains its maximum at $\varepsilon =\frac{1}{s+2\ln s}$ and we must have
$$s\gr -\ln(s+2\ln s)+\left(1-\frac{1}{s+2\ln s}\right)(s+2\ln s)=s+2\ln s-\ln(s+2\ln s)-1.$$
It follows that $\ln(s+2\ln s)+1\gr 2\ln s$, which implies that $e(s+2\ln s)\gr s^2$, a contradiction if
$s\gr s_0$ for a large enough absolute constant $s_0>0$.

We have thus shown that if $s\gr s_0$ and $x\in B_s(\mu)$ then $\varphi_{\mu}(x)\gr e^{-s-3\ln s}$, i.e.
$B_s(\mu)\subseteq T_{s+3\ln s}(\mu)$. The inclusion $T_s(\mu)\subseteq B_s(\mu)$ is \eqref{eq:floating-1} above.
\end{proof}

The next simple lemma compares the families $\{T_s(\mu)\}_{t>0}$ and $\{Z_t^+(\mu)\}_{t>0}$ of floating bodies and $L_t$-centroid bodies.

\begin{lemma}\label{lem:floating-2}Let $\mu$ be a centered log-concave probability measure on $\mathbb{R}^n$.
For any $\delta>0$ and any $t\gr 1$ we have that
$$T_{t\ln(1+\delta)}(\mu)\subseteq (1+\delta)Z_t^+(\mu).$$
\end{lemma}

\begin{proof}Assume that $x\notin (1+\delta)Z_t^+(\mu)$. Then, we may find $\delta^{\prime}>\delta$ and $\xi\in S^{n-1}$ such that
$x\in H=\{y\in\mathbb{R}^n:\langle y,\xi\rangle \gr (1+\delta^{\prime})h_{Z_t^+}(\xi)\}$. From Markov's inequality we see
that
$$\mu(H)\ls (1+\delta^{\prime})^{-t} <(1+\delta)^{-t}=e^{-t\ln(1+\delta)},$$
therefore, $\varphi_{\mu}(x)\ls\mu(H)< e^{-t\ln(1+\delta)}$, which shows that $x\notin T_{t\ln(1+\delta)}(\mu)$.
The lemma follows.\end{proof}

\section{Moments of the Cram\'{e}r transform}\label{section-3}

In this section we prove Theorem~\ref{th:moments} and Theorem~\ref{th:small-moments}. We start with the observation that,
without loss of generality, we can restrict our attention to isotropic log-concave probability measures. Indeed, a simple computation
shows that if $\mu$ and $\nu$ are two centered log-concave probability measures and $\nu =T_{\ast}\mu$ for some $T\in GL(n)$
then $\Lambda_{\nu}(\xi)=\Lambda_{\mu}(T^t\xi)$ for all $\xi\in\mathbb{R}^n$, and hence, by the definition of the
Legendre transform, we have that
$$\Lambda_{\nu}^{\ast}(x)=\Lambda_{\mu}^{\ast}(T^{-1}x)$$
for all $x\in\mathbb{R}^n$. Then, from \eqref{eq:push-forward} we get
\begin{equation}\label{eq:push-forward-2}\int_{{\mathbb R}^n}g(\Lambda_{\nu}^{\ast}(x))d\nu (x)=\int_{{\mathbb R}^n}g(\Lambda_{\mu}^{\ast}(y))d\mu (y)\end{equation}
for every bounded Borel measurable function $g:{\mathbb R}^n\to {\mathbb R}$. Since every centered log-concave
probability measure $\mu$ on $\mathbb{R}^n$ has an isotropic push forward $\nu=T_{\ast}\mu$, where $T\in GL(n)$,
we may check the assertion of both theorems for $\nu$, and then it also holds true fror $\mu$.

We shall also use the next simple but useful lemma.

\begin{lemma}\label{lem:2}Let $\mu$ be a centered log-concave probability measure on ${\mathbb R}^n$. For
any $\delta>0$ and any Borel subset $A$ of $\mathbb{R}^n$ we have that
$$\mu((1+\delta)A)\ls e^{2n\delta }\mu(A).$$
\end{lemma}

\begin{proof}Note that if $A\subset {\mathbb R}^n$ is a Borel set, then
$$\mu((1+\delta)A)=\int_{(1+\delta)A}f_{\mu}(x)\,dx=(1+\delta)^n\int_Af_{\mu}((1+\delta)x)\,dx.$$
Since $f_{\mu}$ is log-concave, we see that
$$f_{\mu}((1+\delta)x)\ls f_{\mu}(x)\left(\frac{f_{\mu}(x)}{f_{\mu}(0)}\right)^{\delta}\ls e^{n\delta}f_{\mu}(x)$$
for every $x\in {\mathbb R}^n$, because $f_{\mu}(x)\ls e^nf_{\mu}(0)$ by \eqref{eq:frad-2}. It follows that
\begin{equation*}\mu((1+\delta)A)\ls (1+\delta)^ne^{n\delta}\mu(A)\ls e^{2n\delta}\mu(A)\end{equation*}
as claimed. \end{proof}

A weak integrability result can be obtained if we combine Proposition~\ref{prop:1} with the next technical proposition (see \cite[Proposition~5.6]{BGP-depth} for a proof).

\begin{proposition}\label{prop:3-weak}Let $\mu $ be an isotropic log-concave probability measure on ${\mathbb R}^n$. For any $\delta\in (0,1)$
and any $t\gr C_{\delta }n\ln n$ we have that $$\mu ((1+\delta )Z_t^+(\mu ))\gr 1-e^{-c\delta t}$$
where $C_{\delta }=C\delta^{-1}\ln\left(2/\delta\right)$ and $C,c>0$ are absolute positive constants.
\end{proposition}

\begin{theorem}\label{th:moments-weak}Let $\mu$ be a centered log-concave probability measure on ${\mathbb R}^n$.
For every $0<p<1$ we have that
$$\int_{{\mathbb R}^n}|\Lambda_{\mu}^{\ast }(x)|^p\,d\mu(x)<\infty .$$
\end{theorem}

\begin{proof}We may assume that $\mu$ is isotropic. Let $0<p<q<1$. Let $t>1$ that will be chosen appropriately large
and apply Proposition~\ref{prop:3-weak} with $\delta =\frac{\ln t}{ct^q}$. Then, we get
\begin{equation}\label{eq:moments-1}\mu \left(\left(1+\frac{\ln t}{ct^q}\right)Z_t^+(\mu )\right)\gr 1-e^{-t^{1-q}\ln t}\end{equation}
provided that $t\gr C_1(q)t^qn\ln n$, or equivalently $t\gr t_n=C_2(q)(n\ln n)^{\frac{1}{1-q}}$.
Now, let $t\gr t_n$. From Proposition~\ref{prop:1} we see that
$$ \left(1+\frac{\ln t}{ct^q}\right)Z_t^+(\mu )\subseteq \left(1+\frac{\ln t}{ct^q}\right)
\left(1+\frac{\ln t}{t}\right)B_t(\mu)$$
and applying Lemma~\ref{lem:2} twice we obtain
\begin{align}\label{eq:moments-2}\mu \left(\left(1+\frac{\ln t}{ct^q}\right)Z_t^+(\mu )\right)
&\ls \exp\left(\frac{2n\ln t}{ct^q}+\frac{2n\ln t}{t}\right)\mu(B_t(\mu))\ls \exp\left(\frac{4n\ln t}{ct^q}\right)\mu(B_t(\mu))\\
\nonumber &\ls \left(1-\frac{4n\ln t}{ct^q}\right)^{-1}\mu(B_t(\mu)).\end{align}
From \eqref{eq:moments-1} and \eqref{eq:moments-2} we get
$$\mu(B_t(\mu))\gr \left(1-e^{-t^{1-q}\ln t}\right)\left(1-\frac{4n\ln t}{ct^q}\right)\gr
1-\frac{8n\ln t}{ct^q}$$
if $t\gr s_n$ where $s_n\gr t_n$ is large enough and depends only on $n$ (note that
$\exp(t^{1-q}\ln t)\gr ct^q/(4n\ln t)$ for large enough $t$, independently from $n$).
Now, we write
\begin{align*}\int_{\mathbb{R}^n}|\Lambda_{\mu}^{\ast}(x)|^pd\mu(x)
&=\int_0^{\infty}pt^{p-1}\mu(\{x:\Lambda_{\mu}^{\ast}(x)>t\})\,dt =\int_0^{\infty}pt^{p-1}(1-\mu(B_t(\mu)))\,dt\\
&=\int_0^{s_n}pt^{p-1}(1-\mu(B_t(\mu)))\,dt+\int_{s_n}^{\infty}pt^{p-1}(1-\mu(B_t(\mu)))\,dt\\
&\ls \int_0^{s_n}pt^{p-1}\,dt+\frac{8np}{c}\int_{s_n}^{\infty}t^{p-1}\frac{\ln t}{t^q}\,dt\\
&=s_n^p+\frac{8np}{c}\int_{s_n}^{\infty}\frac{\ln t}{t^{1+q-p}}dt<+\infty\end{align*}
where the last integral converges because $1+q-p>1$.
\end{proof}

We have already described in the introduction that the main ingredient for the stronger integrability
estimate of Theorem~\ref{th:moments} is a lemma which compares the families $\{R_t(\mu)\}_{t>0}$
and $\{B_t(\mu)\}_{t>0}$. This is the content of Lemma~\ref{lem:level-B} below, which is then combined
with Proposition~\ref{prop:r-2}. For the proof of Lemma~\ref{lem:level-B} we need a technical fact; if $\mu$ is
isotropic and $t$ is large enough then $R_t(\mu)$ contains a constant multiple of the Euclidean unit ball.
The proof of the next lemma is essentially contained in \cite[Lemma~5.4]{Klartag-2007clt}. 

\begin{lemma}\label{lem:gamma-ball}Let $\mu$ be an isotropic log-concave probability measure on ${\mathbb R}^n$.
For any $t\gr 20n$ we have that
$$R_t(\mu)\supseteq  \frac{1}{3}B_2^n.$$
\end{lemma}

\begin{proof}Let $r:=r(R_t(\mu))$ denote the inradius of $R_t(\mu)$ with respect to the origin. Then there exists $\xi\in S^{n-1}$ such that
$h_{R_t(\mu)}(\xi)\ls r$, and consequently
$$R_t(\mu)\subseteq \{x\in\mathbb{R}^n:\langle x,\xi\rangle \ls r\}.$$
We decompose this half-space as
$$\{x\in\mathbb{R}^n:\langle x,\xi\rangle \ls r\}=\{x\in\mathbb{R}^n:\langle x,\xi\rangle \ls 0\}\cup \{x\in\mathbb{R}^n:0<\langle x,\xi\rangle \ls r\}.$$
Since $\mu$ is centered, Gr\"{u}nbaum's lemma (see~\cite[Lemma~2.2.6]{BGVV-book}) yields 
$$\mu(\{x\in\mathbb{R}^n:\langle x,\xi\rangle\ls 0\})\ls 1-\frac{1}{e}.$$
Let $F_{\xi}=\{t\xi:t\in\mathbb{R}\}$ be the one-dimensional subspace spanned by $\xi$.
Since $\mu$ is isotropic, the one-dimensional marginal $$g_{\xi}(t)=(\pi_{F_{\xi}}(f))(t)=\int_{F_{\xi}^{\perp}}f(y+t\xi )\,dy$$
is an isotropic log-concave density on $\mathbb{R}$ (see~\cite[Proposition~5.1.11]{BGVV-book}). Consequently,
\begin{equation*}\|g_{\xi}\|_{\infty}=L_{g_{\xi}}\ls 1.\end{equation*}
Therefore,
$$\mu(\{x\in\mathbb{R}^n:0<\langle x,\xi\rangle\ls r\})=\int_0^r\left(\int_{F_{\xi}^{\perp}}f(y+t\xi)\,dy\right)\,dt=\int_0^rg_{\xi}(t)\,dt\ls r.$$
Combining these estimates and using Proposition~\ref{prop:r-2}, we obtain
$$1-e^{-t/4}\ls \mu(R_t(\mu))\ls 1-\frac{1}{e}+r.$$
This implies 
$$r\gr e^{-1}-e^{-t/5}\gr e^{-1}-e^{-4}\gr 1/3,$$
as claimed.\end{proof}

We are now ready to compare $\{R_t(\mu)\}_{t>0}$ and $\{T_t(\mu)\}_{t>0}$.

\begin{lemma}\label{lem:level-B}Let $\mu$ be an isotropic log-concave probability measure on ${\mathbb R}^n$. For
any $\delta\in (0,1)$ and $t\gr n\ln n$ we have that
$$(1-\delta)R_t(\mu)\subseteq T_{g(t,\delta)}$$
where $g(t,\delta)\ls 2t+n\ln(1/\delta)$. In particular, if $\delta =e^{-t/2}$ we get
$$(1-e^{-t/2})R_t(\mu)\subseteq T_{nt}(\mu).$$
Note that the same inclusions hold with $T_{g(t,\delta)}$ replaced by $B_{g(t,\delta)}$, by~\eqref{eq:floating-1}.
\end{lemma}

\begin{proof}From Lemma~\ref{lem:gamma-ball} we know that $R_t(\mu)\supseteq \frac{1}{3}B_2^n$. Let $x\in (1-\delta)R_t(\mu)$.
Since $(\delta/3)B_2^n\subseteq \delta R_t(\mu)$ we have that
$$B(x,\delta/3):=x+(\delta/3)B_2^n\subseteq (1-\delta)R_t(\mu)+\delta R_t(\mu)=R_t(\mu).$$
Then, if $H$ is a closed half-space such that $x\in\partial(H)$ we have that $H\cap B(x,\delta/3)\subseteq R_t(\mu)$,
and hence $f_{\mu}(y)\gr e^{-t}f_{\mu}(0)$ for all $y\in H\cap B(x,\delta/3)$. Since $H\cap B(x,\delta/3)$
is half of a ball of radius $\delta/3$, it follows that
$$\mu(H)\gr \int_{H\cap B(x,\delta/3)}f_{\mu}(y)\,d\mu(y)\gr e^{-t}f_{\mu}(0)\,|H\cap B(x,\delta/3)|
=e^{-t}f_{\mu}(0)\,\frac{\omega_n}{2}(\delta/3)^n.$$
Now, recall that $f_{\mu}(0)\gr e^{-n}\|f_{\mu}\|_{\infty}=e^{-n}L_{\mu}^n\gr (c_1/e)^n$ because $L_{\mu}\gr c_1$
for some absolute constant $c_1>0$, and $\frac{\omega_n}{2}\gr \left(\frac{c_2}{n}\right)^{n/2}$ for some absolute
constant $c_2>0$. Since $H$ was arbitrary, by the definition of $\varphi_{\mu}$ we get
$$\varphi_{\mu}(x)\gr e^{-t}\left(\frac{c_3}{n}\right)^{n/2}\delta^n=e^{-g(t,\delta)},$$
with $c_3=c_1^2c_2/(9e^2)$, where
\begin{equation}\label{eq:delta}g(t,\delta)=t+\frac{n}{2}\ln(c_4n)+n\ln(1/\delta)\end{equation}
with $c_4=1/c_3$. Since $t\gr n\ln n$, we see that $\frac{n}{2}\ln(c_4n)\ls t$ and the first assertion of the lemma follows.
For the second assertion, note that $g(t,e^{-t/2})\ls 2t+\frac{n}{2}t=\frac{n+4}{2}t\ls nt$ if $n\gr 4$.
\end{proof}

Lemma~\ref{lem:level-B} shows that if $t\gr n\ln n$ then $R_t(\mu)\subseteq \frac{1}{1-\delta}B_{g(t,\delta)}(\mu)
\subseteq (1+2\delta)B_{g(t,\delta)}(\mu)$ provided that $0<\delta \ls 1/2$. Using also Lemma~\ref{lem:2} we can
prove Theorem~\ref{th:moments}.

\begin{proof}[Proof of Theorem~$\ref{th:moments}$]As we observed in the beginning of this section,
we may assume that $\mu$ is isotropic. Let $t_n:=n\ln n$. From Lemma~\ref{lem:level-B} and Proposition~\ref{prop:floating-1} we know that
if $t\gr t_n$ then
$$R_t(\mu)\subseteq (1+2e^{-t/2})T_{nt}(\mu)\subseteq (1+2e^{-t/2})B_{nt}(\mu).$$
Moreover, Proposition~\ref{prop:r-2} shows that $\mu(R_t(\mu ))\gr 1-e^{-t/4}$. Combining these facts
with Lemma~\ref{lem:2} we get
$$1-e^{-t/4}\ls \mu(R_t(\mu)) \ls \exp\left(4n\,e^{-t/2}\right)\mu(B_{nt}(\mu))\ls (1+8ne^{-t/2})\mu(B_{nt}(\mu)),$$
which finally gives
$$\mu(B_{nt}(\mu))\gr 1-e^{-t/8}$$
if $n\gr n_0$ for some fixed $n_0\in\mathbb{N}$. Now, we write
\begin{align*}\int_{\mathbb{R}^n}\exp\left(\Lambda_{\mu}^{\ast}(x)/(16n)\right)d\mu(x)
&=1+\frac{1}{16}\int_0^{\infty}e^{\frac{t}{16}}\mu(\{x\in {\mathbb R}^n:\Lambda_{\mu}^{\ast}(x)>nt\})\,dt\\
&\ls 1+\frac{1}{16}\int_0^{t_n}e^{\frac{t}{16}}\,dt
+\frac{1}{16}\int_{t_n}^{\infty}e^{\frac{t}{16}}(1-\mu(B_{nt}(\mu)))\,dt\\
&\ls e^{\frac{t_n}{16}}+\frac{1}{16}\int_{t_n}^{\infty}e^{\frac{t}{16}}e^{-\frac{t}{8}}dt
=e^{\frac{t_n}{16}}+\frac{1}{16}\int_{t_n}^{\infty}e^{-\frac{t}{16}}dt \ls 2e^{\frac{t_n}{16}}.
\end{align*}
Therefore, we have the assertion of the theorem with $c=\frac{1}{16}$.\end{proof}

\begin{remark}\rm We can obtain an upper bound on the Orlicz norm
$$\|\Lambda_{\mu}^{\ast}\|_{L^{\psi_1}(\mu)}:=\inf\left\{r>0: \int_{\mathbb{R}^n}\exp\left(\Lambda_{\mu}^{\ast}(x)/r\right)d\mu(x)\ls 2\right\}$$
corresponding to $\psi_1(t)=e^t-1$, using H\"{o}lder's inequality. We write
$$\int_{\mathbb{R}^n}\exp\left(\Lambda_{\mu}^{\ast}(x)/r\right)d\mu(x)\ls \left(\int_{\mathbb{R}^n}\exp\left(\Lambda_{\mu}^{\ast}(x)/(16n)\right)d\mu(x)\right)^{\frac{16n}{r}}
\ls 2^{\frac{16n}{r}}e^{\frac{nt_n}{r}}$$
for any $r>16n$, and then choosing $r_0=\frac{32}{\ln 2}nt_n$ we check that $\|\Lambda_{\mu}^{\ast}\|_{L^{\psi_1}(\mu)}\ls r_0\approx n^2\ln n$.   
\end{remark}

For the proof of Theorem~\ref{th:small-moments} we use Lemma~\ref{lem:level-B} again, with a different
choice of $\delta$ depending on $t$. In fact we can estimate $\|\Lambda_{\mu}^{\ast}\|_{L^p(\mu)}$ for all $1\ls p\ls cn$.

\begin{theorem}\label{th:small-moments-p}For every centered log-concave probability measure $\mu$ on ${\mathbb R}^n$
and any $1\ls p\ls c_1n$ we have that
$$\|\Lambda_{\mu}^{\ast}\|_{L^p(\mu)}\ls c_2pn\ln n$$
where $c_1,c_2>0$ are absolute constants.
\end{theorem}

\begin{proof}As in the previous proof, we may assume that $\mu$ is isotropic.
Let $t_n:=n\ln n$. Applying Proposition~\ref{prop:floating-1} and Lemma~\ref{lem:level-B} with
$\delta=t^{-p-2}$, or more precisely using \eqref{eq:delta}, we see that
\begin{equation}\label{eq:for-theorem-4.2-1}(1-t^{-p-2})R_t(\mu)\subseteq T_{g(t,t^{-p-2})}(\mu)\subseteq B_{g(t,t^{-p-2})}(\mu)\end{equation}
where 
\begin{equation}\label{eq:for-theorem-4.2-2}g(t,t^{-(p+2)})=t+\frac{n}{2}\ln(c_4n)+(p+2)n\ln t\ls 2(p+3)t\end{equation}
because the function $m(t)=\frac{t}{\ln t}$ is increasing on $[e,\infty)$ and hence
$\frac{n}{2}\ls \frac{n\ln n}{\ln(n\ln n)} \ls \frac{t}{\ln t}$ for all $t\gr t_n$, which
implies that $n\ln t\ls 2t$ for all $t\gr t_n$. From Lemma~\ref{lem:2} it follows that
$$\mu(B_{2(p+3)t})(\mu)\gr e^{-4n/t^{p+2}}\mu(R_t(\mu))\gr  e^{-4n/t^{p+2}}(1-e^{-t/4})\gr 1-\frac{8n}{t^{p+2}}>1-\frac{1}{t^{p+1}}$$
for all $t\gr t_n$ and all $0<p\ls n/8$ (here we use the fact that $\frac{t}{4}>(p+2)\ln t$ if $t\gr t_n$ and $p\ls n/8$,
which implies that $e^{-t/4}<2n/t^{p+2}$). Then, we write
\begin{align*}\int_{\mathbb{R}^n}|\Lambda_{\mu}^{\ast}(x)/2(p+3)|^pd\mu(x)
&=\int_0^{\infty}pt^{p-1}\mu(\{x\in {\mathbb R}^n:\Lambda_{\mu}^{\ast}(x)>2(p+3)t\})\,dt\\
&=\int_0^{\infty}pt^{p-1}(1-\mu(B_{2(p+3)t}(\mu)))\,dt\\
&\ls t_n^p+\int_{t_n}^{\infty}pt^{p-1}\,\frac{1}{t^{p+1}}\,dt
\ls t_n^p+\frac{p}{t_n}.
\end{align*}
It follows that $\|\Lambda_{\mu}^{\ast}\|_{L^p(\mu)}\ls cpt_n=cpn\ln n$, for some absolute constant $c>0$.
\end{proof}

\begin{remark}\label{rem:optimal}\rm The upper bound of Theorem~\ref{th:small-moments} is sharp: a computation in
\cite{BGP-threshold} shows that
$$\|\Lambda_{\mu_{D_n}}^{\ast}\|_1\approx \|\Lambda_{\mu_{D_n}}^{\ast}\|_2\approx n\ln n$$
where $D_n$ is the centered Euclidean ball of volume $1$ in $\mathbb{R}^n$ and $\mu_{D_n}$ is the
uniform measure on $D_n$. On the other hand, it was proved in \cite{BGP-depth} that if $\mu$ is a log-concave
probability measure on ${\mathbb R}^n$, $n\gr n_0$, then
$$\int_{\mathbb{R}^n}e^{-\Lambda_{\mu }^{\ast}(x)}\,d\mu(x) \ls \exp\left(-cn/L_{\mu }^2\right).$$
From Jensen's inequality we immediately get
$$\|\Lambda_{\mu}^{\ast}\|_1=\int_{\mathbb{R}^n}\Lambda_{\mu }^{\ast}(x)\,d\mu(x) \gr cn/L_{\mu }^2\gr c_1n.$$
This lower bound is also optimal, as one can check from the example of the uniform measure on the cube $C_n=\left[-\frac{1}{2},\frac{1}{2}\right]^n$. Summarizing, for any centered log-concave probability measure
$\mu$ on $\mathbb{R}^n$ we have that
$$c_1n\ls \|\Lambda_{\mu}^{\ast}\|_{L^1(\mu)}\ls\|\Lambda_{\mu}^{\ast}\|_{L^2(\mu)}\ls c_2n\ln n,$$
where $c_1,c_2>0$ are absolute constants.
\end{remark}

\section{Applications and further remarks}\label{section-4}

In this section we provide a number of applications of our approach. In particular, Theorem~\ref{th:non-sharp}
establishes an asymptotically best possible uniform lower threshold for the expect measure of random polytopes
with vertices that have a log-concave distribution.

\subsection{Uniform thresholds for the measure of random polytopes}\label{subsection-4.1}

In this subsection we prove Theorem~\ref{th:rough}. We start with a short overview of related results that should be
compared with our new bound. Uniform upper and lower thresholds were established by Chakraborti, Tkocz and Vritsiou in \cite{Chakraborti-Tkocz-Vritsiou-2021} in the case where $\mu $ is an even log-concave or $\kappa$-concave probability
measure supported on a convex body $K$ in ${\mathbb R}^n$. Assuming that $\mu$ is log-concave and $X_1,X_2,\ldots $ are independent random points distributed according to $\mu $, for any $n<N\ls \exp (c_1n/L_{\mu }^2)$ we have that
\begin{equation}\label{eq:rough-1}{\mathbb E}_{\mu^N}\big(|K_N|/|K|\big) \ls \exp\left(-c_2n/L_{\mu }^2\right),\end{equation}
where $K_N={\rm conv}\{X_1,\ldots ,X_N\}$ and $c_1,c_2>0$ are absolute constants. In the same work it is
shown that if $\mu $ is assumed $\kappa $-concave then for any $M\gr C$ and any $N\gr \exp\left(\frac{1}{\kappa }(\ln n+2\ln M)\right)$ we have that
\begin{equation}\label{eq:rough-2}{\mathbb E}_{\mu^N}\big(|K_N|/|K|\big)\gr 1-1/M,\end{equation}
where $C>0$ is an absolute constant. Afterwards, the same question was studied in \cite{BGP-depth} for $0$-concave, i.e. log-concave, probability measures. The upper threshold in \cite{BGP-depth} states that there exists an absolute constant $c>0$ such that if $N_1(n)=\exp (cn/L_n^2)$ then
\begin{equation}\label{eq:rough-3}\sup_{\mu }\Big(\sup\Big\{{\mathbb E}_{\mu^N}[\mu (K_N)]:N\ls N_1(n)\Big\}\Big)\longrightarrow 0\end{equation}
as $n\to\infty $, where the first supremum is over all log-concave probability measures $\mu $ on ${\mathbb R}^n$.
Regarding the lower threshold, it was first proved in \cite{BGP-depth} that, for any $\delta\in (0,1)$,
\begin{equation}\label{eq:rough-4}\inf_{\mu }\Big(\inf\Big\{ {\mathbb E}_{\mu^N}\big[\mu ((1+\delta )K_N)\big]: N\gr \exp \big (C\delta^{-1}\ln \left(2/\delta \right)n\ln n\big )\Big\}\Big)\longrightarrow 1\end{equation} as $n\to\infty $, where the first infimum is over all log-concave probability measures $\mu $ on ${\mathbb R}^n$ and $C>0$ is an absolute constant.
Using Lemma~\ref{lem:2}, from \eqref{eq:rough-4} one can deduce that
\begin{equation}\label{eq:rough-5}\inf_{\mu }\Big(\inf\Big\{ {\mathbb E}\,\big[\mu (K_N)\big]: N\gr \exp (C(n\ln n)^2u(n))\Big\}\Big)\longrightarrow 1\end{equation}
as $n\to\infty $, where $C>0$ is an absolute constant, the first infimum is over all log-concave probability measures
$\mu $ on ${\mathbb R}^n$ and $u(n)$ is any function with $u(n)\to\infty $ as $n\to\infty $.

The proof of \eqref{eq:rough-4} is based on Proposition~\ref{prop:3-weak}. We
can obtain a variant of this fact, with a different proof and a slightly better dependence on $\delta$.

\begin{proposition}\label{prop:3}Let $\mu $ be a centered log-concave probability measure on ${\mathbb R}^n$.
For any $\delta\in \left(\frac{3}{n},1\right)$
and any $t\gr \frac{c_1}{\delta}n\ln n$ we have that
$$\mu ((1+\delta )Z_t^+(\mu ))\gr 1-e^{-c_2\delta t}$$
where $c_1,c_2>0$ are absolute positive constants.
\end{proposition}

\begin{proof}Proposition~\ref{prop:B<Z} shows that for any $m\gr 1$ and any $\eta\in (0,1)$ we have that
\begin{equation*}B_m(\mu) \subseteq (1+\eta)Z_{cm/\eta}^+(\mu)\end{equation*}
where $c>0$ is an absolute constant. Now, from Lemma~\ref{lem:level-B} and Proposition~\ref{prop:floating-1} we know that if $s\gr n\ln n$ then
$$(1-\eta)R_s(\mu)\subseteq B_{g(s,\eta)}$$
where $g(s,\eta)\ls 2s+n\ln(1/\eta)$. Assuming that $\frac{1}{n}<\eta<\frac{1}{3}$ we see that
$g(s,\eta)\ls 3s$ and hence
$$R_s(\mu)\subseteq\frac{1}{1-\eta}B_{3s}\subseteq \frac{1+\eta}{1-\eta}Z_{3cs/\eta}^+(\mu)
\subseteq (1+3\eta)Z_{3cs/\eta}^+(\mu)$$
for every $s\gr n\ln n$, because $\eta<\frac{1}{3}$ implies that $\frac{1+\eta}{1-\eta}\ls 1+3\eta$.
Since $s\gr 5n$, we may also apply Proposition~\ref{prop:r-2} to get $\mu(R_s(\mu ))\gr 1-e^{-s/4}$.
It follows that
$$\mu((1+3\eta)Z_{3cs/\eta}^+(\mu))\gr 1-e^{-s/4}$$
for every $s\gr n\ln n$ and any $\eta\in (1/n,1/3)$. Setting $\delta=3\eta$ and $t=3cs/\eta=9cs/\delta$
we get the assertion of the proposition. \end{proof}

Having established Proposition~\ref{prop:3} and following the proofs of Theorem~5.5 and Theorem~5.8 from
\cite{BGP-depth} we can check that
\begin{equation}\label{eq:rough-6}\inf_{\mu }\Big(\inf\Big\{ {\mathbb E}_{\mu^N}\big[\mu ((1+\delta )K_N)\big]: N\gr \exp \big (C\delta^{-1}n\ln n\big )\Big\}\Big)\longrightarrow 1\end{equation}
and then
\begin{equation}\label{eq:rough-7}\inf_{\mu }\Big(\inf\Big\{ {\mathbb E}\,\big[\mu (K_N)\big]: N\gr \exp (Cn^2(\ln n)u(n))\Big\}\Big)\longrightarrow 1\end{equation}
as $n\to\infty$. However, using directly the family $\{T_t(\mu)\}_{t>0}$ of floating bodies of $\mu$ instead of
the family $\{Z_t^+(\mu)\}_{t>0}$ of centroid bodies of $\mu$, we can give an alternative
proof of the uniform lower threshold with an optimal dependence on the dimension.

\begin{theorem}\label{th:non-sharp}There exists an absolute constant $C>0$ such that
$$\inf_{\mu }\Big(\inf\Big\{ {\mathbb E}_{\mu^N}\big[\mu (K_N)\big]: N\gr \exp (Cn\ln n)\Big\}\Big)\longrightarrow 1$$
as $n\to\infty $, where the first infimum is over all log-concave probability measures $\mu $ on ${\mathbb R}^n$.
\end{theorem}

\begin{proof}Let $\mu$ be a log-concave probability measure on $\mathbb{R}^n$. Since the expectation
${\mathbb E}_{\mu^N}\big[\mu (K_N)\big]$ is a affinely invariant quantity, we may assume that $\mu$ is isotropic.
Let $t_n:=n\ln n$. In the proof of Theorem~\ref{th:small-moments} (see \eqref{eq:for-theorem-4.2-1} and \eqref{eq:for-theorem-4.2-2}) we saw that
$$\mu(T_{10t}(\mu))>1-1/t^3$$
for all $t\gr t_n$. By the definition of the family $\{T_t(\mu)\}_{t>0}$, for any $x\in T_{10t}(\mu)$ we have
\begin{equation*}\varphi_{\mu }(x)\gr e^{-10t}.\end{equation*}
We use the following standard lemma (which is stated in this form in \cite[Lemma~3]{Chakraborti-Tkocz-Vritsiou-2021};
for a proof see \cite{DFM} or \cite[Lemma~4.1]{Gatzouras-Giannopoulos-2009}): For every Borel subset $A$ of ${\mathbb R}^n$ we have that
$$1-\mu^N(K_N\supseteq A)\ls 2\binom{N}{n}\left (1-\inf_{x\in A}\varphi_{\mu }(x)\right)^{N-n}.$$
Therefore,
$${\mathbb E}_{\mu^N}[\mu (K_N)]\gr \mu (A)\left (1-2\binom{N}{n}\left (1-\inf_{x\in A}\varphi_{\mu }(x)\right)^{N-n}\right).$$
Setting $A=T_{10t}(\mu)$ we get
\begin{align*}\mu^N\Big(K_N\supseteq T_{10t}(\mu)\Big)
&\gr 1-2\binom{N}{n}\left [1-e^{-10t}\right]^{N-n}\\
&\gr 1-\left(\frac{2eN}{n}\right)^n\exp \left(-(N-n)e^{-10t}\right).
\end{align*}
This last quantity tends to $1$ as $n\to\infty $ if
\begin{equation}\label{eq:conditionC1}n\ln (4eN/n)<(N-n)e^{-10t},\end{equation}
and we easily check that \eqref{eq:conditionC1} holds true if $t\gr t_n$ and $N\gr \exp (Ct)$ for a large enough absolute constant $C>0$. 
Therefore, if $N\gr \exp (Cn\ln n)$ we see that
\begin{align*}{\mathbb E}_{\mu^N}\left[\mu (K_N)\right]
&\gr \mu (T_{10t_n}(\mu))\times \mu^N\Big(K_N\supseteq T_{10t_n}(\mu)\Big)\\
&\gr \big(1-t_n^{-3}\big)\left[1-\left(\frac{2eN}{n}\right)^n\exp \left(-(N-n)e^{-10t_n}\right)\right]\longrightarrow 1
\end{align*}
as $n\to\infty $.
\end{proof}

\subsection{Distribution of the half-space depth}\label{subsection-4.2}

It was proved in \cite{BGP-depth} that if $\mu$ is a log-concave probability measure on ${\mathbb R}^n$, $n\gr n_0$, then
$$\exp(-c_1n)\ls {\mathbb E}_{\mu }(\varphi_{\mu }) \ls \exp\left(-c_2n/L_{\mu}^2\right)\ls\exp(-c_3n)$$
where $L_{\mu }$ is the isotropic constant of $\mu $ and $c_i>0$, $n_0\in {\mathbb N}$ are absolute constants.
In this subsection we discuss the question to determine the values of $p>0$ for which
${\mathbb E}_{\mu }(\varphi_{\mu }^{-p})$ is finite. Brazitikos and Chasapis have shown in \cite[Proposition~3.2]{Brazitikos-Chasapis-2024}
that in the $1$-dimensional case one has ${\mathbb E}(\varphi_{\mu}^{-p})\ls 2^p/(1-p)<\infty $
for all $0<p<1$ and any probability measure $\mu$ on $\mathbb{R}$.

A simple computation with the standard Gaussian measure $\gamma_n$ on $\mathbb{R}^n$ shows that some restriction
on $p$ cannot be avoided. Using the rotational invariance of $\gamma_n$ we easily check that
$$\varphi_{\gamma_n}(x)=1-\Phi(|x|)=\frac{1}{\sqrt{2\pi}}\int_{|x|}^{\infty}e^{-t^2/2}dt$$
where $|x|$ denotes Euclidean norm. This implies that $\varphi_{\gamma_n}(x)\ls \frac{1}{\sqrt{2\pi}|x|}e^{-|x|^2/2}$,
and hence
$$J_{\gamma_n}(p):=\int_{{\mathbb R}^n}\frac{1}{\varphi_{\gamma_n}^p(x)}\,d\gamma_n(x)\gr (2\pi)^{p/2}\int_{{\mathbb R}^n}|x|^pe^{-(1-p)|x|^2/2}\,dx.$$
It follows that $J_{\gamma_n}(p)<\infty$ for all $0<p<1$ but $J_{\gamma_n}(1)=\infty$.

Our results allow us to show that there exists an absolute constant $c>0$ such that
$$J_{\mu}(p):=\int_{{\mathbb R}^n}\frac{1}{\varphi_{\mu}^p(x)}\,d\mu (x)<\infty$$
for any $0<p\ls c/n$ and every log-concave probability measure $\mu$ on $\mathbb{R}^n$.

\begin{proof}[Proof of Theorem~$\ref{th:negative-phi}$]We may assume that $\mu$ is centered.
The theorem follows immediately if we combine Theorem~\ref{th:moments}
with the inequality
$$\Lambda_{\mu}^{\ast}(x)\gr \ln\left(\frac{\varepsilon}{(2\varphi_{\mu}(x))^{1-\varepsilon}}\right).$$
for every $x\in {\rm supp}(\mu)$ and any $\varepsilon\in (0,1)$ (this is a result of Brazitikos and Chasapis from \cite{Brazitikos-Chasapis-2024} that we have already used in the proof of Proposition~\ref{prop:floating-1}).
Choosing $\varepsilon=1/2$ we get
$$\frac{1}{2^{3/2}\varphi_{\mu}(x)^{1/2}}\ls e^{\Lambda_{\mu}^{\ast}(x)}$$
and hence
$$J_{\mu}(p):=\int_{{\mathbb R}^n}\frac{1}{\varphi_{\mu}^p(x)}\,d\mu (x)\ls
2^{3p}{\mathbb E}\,\big[e^{2p\Lambda_{\mu}^{\ast}(x)}\big]<\infty$$
if $2p\ls c/n$ where $c>0$ is the absolute constant from Theorem~\ref{th:moments}.
\end{proof}

\subsection{Affine surface area}\label{subsection-4.3}

We close this article with some remarks on the connection of the integrability properties of $\Lambda_{\mu}^{\ast}$
with the notion of affine surface area. Let us first consider a convex body $K$ in ${\mathbb R}^n$. The affine surface area
of $K$ is defined by
$${\rm as}(K)=\int_{\partial(K)}\kappa(x)^{\frac{1}{n+1}}d\mu_{\partial(K)}(x),$$
where $\kappa(x)$ is the generalized Gauss-Kronecker curvature at $x$ and $\mu_{\partial(K)}$ is the surface measure on
$\partial(K)$ (see \cite{Nagy-Schutt-Werner-2019} and the references therein). The affine isoperimetric inequality states that
$$\left(\frac{{\rm as}(K)}{{\rm as}(B_2^n)}\right)^{n+1}\ls \left(\frac{|K|}{|B_2^n|}\right)^{n-1}$$
with equality if and only if $K$ is an ellipsoid (see \cite[Section~10.5]{Schneider-book}).
Using the fact that ${\rm as}(B_2^n)=n|B_2^n|$ we see that if $|K|=1$ then ${\rm as}(K)\ls c_1$, where
$c_1>0$ is an absolute constant. It is not hard to check that, for every $\delta\in (0,1/2)$, the floating body
$$K_{\delta}=\bigcap\{H^+:H^+\;\hbox{is a closed half-space with}\;|K\cap H^-|=\delta\},$$
where $H^-$ is the complementary half-space of $H^+$ satisfies
$$K_{\delta}=\{x\in\mathbb{R}^n:\varphi_{\mu_K}(x)\gr\delta\}=T_{\ln(1/\delta)}(\mu_K).$$
Sch\"{u}tt and Werner proved in \cite{Schutt-Werner-1990} that for every convex body $K$ in ${\mathbb R}^n$ one has that
$$\lim_{\delta\to 0}\frac{|K|-|K_{\delta}|}{\delta^{\frac{2}{n+1}}}=\frac{1}{2}\left(\frac{n+1}{\omega_{n-1}}\right)^{\frac{2}{n+1}}
{\rm as}(K).$$
In particular, if $|K|=1$ then there exists $\delta_0>0$ such that if $0<\delta <\delta_0$ then
$$1-|K_{\delta}|\ls c_2n\delta^{\frac{2}{n+1}}$$
or, equivalently, there exists $s_0>0$ such that
\begin{equation}\label{eq:T-measure}\sup\{e^{\frac{2s}{n+1}}(1-\mu_K(T_s(\mu_K))):s\gr s_0\}\ls c_2n.\end{equation}
Taking also into account the fact that $T_s(\mu_K)\subseteq B_s(\mu_K)$
for every $s>0$, applying \eqref{eq:T-measure} one can give an alternative proof of the fact that
$$\int_K\exp(\kappa\Lambda_{\mu_K}^{\ast}(x))\,dx <\infty$$
for all $\kappa<\frac{2}{n+1}$. The details appear in \cite[Theorem~6.5]{Giannopoulos-2025}. Theorem~\ref{th:T-measure} is an analogue
of \eqref{eq:T-measure} in the more general setting of log-concave probability measures.

\begin{proof}[Proof of Theorem~$\ref{th:T-measure}$]Let $\mu$ be a log-concave probability measure on $\mathbb{R}^n$.
We may also assume that $\mu$ is centered. In the proof of Theorem~\ref{th:moments} we saw that if $t\gr n\ln n$ then
$R_t(\mu)\subseteq (1+2e^{-t/2})T_{nt}(\mu)$, which implies that $\mu(T_{nt}(\mu))\gr 1-e^{-t/8}$. Equivalently, if $s\gr n^2\ln n$ then
$$\mu(T_s(\mu))\gr  1-e^{-s/(8n)},$$
which shows that $e^{s/(8n)}(1-\mu(T_s(\mu)))\ls 1$. It follows that
$$\sup\{e^{s/(8n)}(1-\mu(T_s(\mu))):s>0\}\ls \exp(n^2\ln n/(8n))$$
and the theorem follows with $c=1/8$ and $c_n=\exp(n\ln n/8)$.
\end{proof}

\bigskip

\noindent {\bf Acknowledgements.} We thank the referee for helpful comments and constructive suggestions that improved the presentation of the paper. 
The first named author acknowledges support by the Hellenic Foundation for
Research and Innovation (H.F.R.I.) in the framework of the call ``Basic research Financing (Horizontal support of all Sciences)” 
under the National Recovery and Resilience Plan ``Greece 2.0” funded by the European Union –NextGenerationEU (H.F.R.I. Project Number: 15445).
The second named author acknowledges support by a PhD scholarship from the National Technical University of Athens.

\bigskip


\footnotesize
\bibliographystyle{amsplain}


\bigskip

\medskip

\medskip

\thanks{\noindent {\bf Keywords:} log-concave probability measures, Cram\'{e}r transform, $L_q$-centroid bodies,
half-space depth, threshold phenomena, random polytopes.}

\smallskip

\thanks{\noindent {\bf 2020 MSC:} Primary 60D05; Secondary 52A23, 52A22, 60E15, 62H05.}

\bigskip

\bigskip

\noindent \textsc{Apostolos \ Giannopoulos}: School of Applied Mathematical and Physical
Sciences, National Technical University of Athens, Department of Mathematics, Zografou Campus, GR-157 80, Athens, Greece.

\smallskip

\noindent \textit{E-mail:} \texttt{apgiannop@math.ntua.gr}

\bigskip

\noindent \textsc{Natalia \ Tziotziou}: School of Applied Mathematical and Physical Sciences, National Technical University of Athens, Department of Mathematics, Zografou Campus, GR-157 80, Athens, Greece.

\smallskip

\noindent \textit{E-mail:} \texttt{nataliatz99@gmail.com}

\end{document}